\documentclass[12pt]{amsart}%
\usepackage{amsfonts}
\usepackage{amsmath}
\usepackage{amssymb}
\usepackage{graphicx}
\usepackage[pagebackref]{hyperref}
\makeatletter
\@addtoreset{equation}{section}
\makeatother

\newtheorem{theorem}{Theorem}[section]

\newtheorem{conjecture}[theorem]{Conjecture}
\newtheorem{corollary}[theorem]{Corollary}

\newtheorem{lemma}[theorem]{Lemma}

\newtheorem{remark}[theorem]{Remark}

\begin{document}
\title[Uniqueness for equations with critical exponential growth]{Uniqueness of positive solutions to elliptic equations with the critical exponential growth on the unit disc and its applications}
\author{Lu Chen, Guozhen Lu, Ying Xue and Maochun Zhu}
\address{School of Mathematics and Statistics, Beijing Institute of Technology, Beijing 100081, P. R. China}
\email{chenlu5818804@163.com}
\address{Department of Mathematics\\
University of Connecticut\\
Storrs, CT 06269, USA}
\email{guozhen.lu@uconn.edu}
\address{School of mathematical sciences\\
Jiangsu University\\
Zhenjiang, 212013, P. R. China\\}
\email{xueying6816@126.com}
\address{School of mathematical sciences\\
Jiangsu University\\
Zhenjiang, 212013, P. R. China\\}
\email{zhumaochun2006@126.com}

\thanks{The first author's research was partly supported by the National Natural Science Foundation of China (No. 12271027). The second author's research was supported partly by the Simons Foundation. The
fourth author's research was supported by National Natural Science Foundation of China (12071185). }

\begin{abstract}
In the past  few decades, uniqueness of positive solutions to elliptic equations with polynomial growth has been extensively studied. However, the corresponding problems associated with the elliptic equation with critical exponential growth
given by the Trudinger-Moser inequalities still remains open. For this kind of equations, the classic non-degeneracy method based on the Pohozaev identity and the study of the linearized equation do not seem to  work. In this paper, we will solve this uniqueness problem. More precisely, we obtain the uniqueness of positive solutions to the following equations:
\begin{equation*}
\begin{cases}
-\Delta u =\lambda ue^{u^2},\quad\quad & x\in B_1\subset \mathbb{R}^2,\\
u>0, & x\in B_1,\ \\
u=0,\quad\quad &x\in \partial B_1,
\end{cases}
\end{equation*}
where $ 0<\lambda<\lambda_1(B_1)$ and $\lambda_1(B_1)$ denotes the first eigenvalue of the operator $-\Delta$ with the Dirichlet boundary in unit disk. Our method relies on delicate and difficult analysis of  radial solutions to the above equation and careful asymptotic expansion of solutions near the boundary. This uniqueness result will shed some light on solving the conjecture that maximizers of the Trudinger-Moser inequality on the unit disc are unique.
  Furthermore, based on this uniqueness result, we develop a new strategy to establish the quantization property of elliptic equations with the critical exponential growth in the balls of hyperbolic spaces, and obtain the multiplicity and non-existence of positive critical points for super-critical Trudinger-Moser functional. Our method for the quantization property and non-existence of the critical points avoids using the complicated blow-up analysis  used in the literature. This method can also be applied to study the similar problems in balls of high dimensional Euclidean space $\mathbb{R}^n$ or hyperbolic spaces provided the uniqueness for the corresponding quasilinear elliptic equations with the critical exponential growth is established.
\end{abstract}

\maketitle {\small {\bf Keywords:} Uniqueness; Critical points and  Multiplicity; Trudinger-Moser exponential growth; Quantization analysis.
 \\

  }

\section{\bigskip Introduction}
The main content of this paper is concerned with the uniqueness theorem for positive solutions of elliptic equations with the Trudinger-Moser growth and  its application in quantization analysis, multiplicity and non-existence of critical points of Trudinger-Moser functional in  balls of Euclidean spaces or hyperbolic spaces. Uniqueness problems and quantization analysis for elliptic equations have attracted much attention due to its importance in applications to PDEs and geometric analysis. Let us first present a brief history of the main results in this direction.
\medskip

In the past few decades, much attention has been paid to the uniqueness of solutions to elliptic equations with the nonlinearity $f$ of  polynomial growth:
\begin{equation}\label{op1}
\begin{cases}
-\Delta u=f(u),\quad\quad & x\in  B_1,\\
u>0,\quad\quad & x\in  B_1, \\
u=0,\quad\quad & x \in \partial B_1,
\end{cases}
\end{equation}
where $B_1$ is the unit ball in $\mathbb{R}^n$ $(n\geq 2)$.  By the classical moving-plane method, one knows that every solution of problem \eqref{op1} must be radial decreasing. Hence the problem \eqref{op1} can be reduced to the following radial equation:

\begin{equation}\label{op2}
		\begin{cases}
			-(r^{n-1}u')'=r^{n-1}f(u),\quad\quad &r\in(0,1), \\
			u>0, \quad\quad & r\in(0,1), \\
			u'(0)= u(1)=0.
		\end{cases}
	\end{equation}
Now, we recall some important results for the nonlinearity $f(u)=\lambda u+u^{p}$, $1<p<+\infty$ and $\lambda\geq 0$ and $n\ge 3$. When $\lambda=0$, $1<p<\frac{n+2}{n-2}$, Gidas, Ni and Nirenberg \cite{GNN} proved that problem \eqref{op1} admits only one radial solution through homogeneity. (By the Pohozaev identity, problem \eqref{op1} does not admit any solutions for $\lambda=0$ and $p\geq \frac{n+2}{n-2}$). When $\lambda>0$ and $1<p\leq \frac{n}{n-2}$, the uniqueness of positive solution was obtained by Ni and Nussbaum \cite{NN}. Kwong and Li \cite{KL} extended this uniqueness result to the case $\lambda>0$ and $1<p<\frac{n+2}{n-2}$, while the uniqueness for the critical case ($p=\frac{n+2}{n-2}$, the Brezis-Nirenberg problem \cite{BN}) was proved by Srikanth \cite{Sri}. In the aforementioned papers, the main idea in proving the uniqueness result is to show that the corresponding linearized equation has only one zero in $(0,1)$. This is the so-called  non-degeneracy method. Subsequently, Adimurthi \cite{Adimurthi} provided an elementary proof for the uniqueness when $\lambda\geq 0$ and $1<p\leq \frac{n+2}{n-2}$ through exploiting a generalized Pohozaev variational identity. For $\lambda>0$ and $p>\frac{n+2}{n-2}$, the uniqueness cannot be expected to hold. Indeed, it has been shown by Budd and Norbury \cite{BuNo} that there exists $\lambda>0$ such that probelm \eqref{op2} has infinitely many solutions when $3\leq n\leq 9 $.

 \medskip
 It should be noted that the critical growth means that the
nonlinearity cannot exceed the polynomial of degree $\frac{n+2}{n-2}$ by the  Sobolev imbedding theorem. While in the case $n=2$, we say that
$f\left(  s\right)  $ has \textit{critical exponential growth} at infinity if
there exists $\alpha_{0}>0$ such that \
\begin{equation}
\underset{t\rightarrow\infty}{\lim}\frac{f\left(  t\right)  }{\exp\left(
\alpha t^{2}\right)  }=%
\genfrac{\{}{.}{0pt}{}{0\text{, \ \ for }\alpha>\alpha_{0}}{+\infty\text{, for
}\alpha<\alpha_{0}},
\label{exponential critical}%
\end{equation}which is given by the famous Trudinger-Moser inequality (\cite{Mo, Tru}):  \begin{equation}\label{Tru}
\sup_{\|\nabla u\|_2^2\leq 1, u\in W^{1,2}_0(\Omega)}\int_{\Omega}\exp(\alpha|u|^{2})dx<\infty, \text{iff } \alpha\leq4\pi,
\end{equation}
where $\Omega\subseteq \mathbb{R}^2$ is a  bounded domain and $W^{1,2}_0(\Omega)$ denotes the usual Sobolev space.
\vskip0.1cm

 Thus, the maximal growth in the case $n=2$ is of exponential type. A natural but nontrivial problem arises:  Can the uniqueness result still hold if we replace the nonlinearity of equation \eqref{op1} with exponential growth, and in particular for $f(t)=\lambda t e^{t^{\mu}}$, with $0<\mu\leq2$ and $\lambda>0$?
\vskip0.1cm

When $\mu=1$, by using a new identity from the beautiful analysis developed by Atkinson and Peletier \cite{AP},  Adimurthi \cite{Adimurthi1} obtained the uniqueness for the subcritical case $f(t)=te^{t}$. Tang \cite{Tang} further showed that the uniqueness is still true for  more general nonlinearity of the type $f(t)=\lambda g(t)e^t$, where $g(t)$ is a polynomial and satisfies  certain conditions.  However, this method  cannot be extended for the case $\mu>1$. Recently, under the assumption that $\|u\|_\infty$ is large enough, Adimurthi, Karthik A and Giacomoni \cite{Adimurthi2} proved the uniqueness of positive radial solution under suitable growth conditions on the nonlinearity. We remark that the nonlinearity in \cite{Adimurthi2} includes the subcritical case $1<\mu<2$ and partially critical case such as $f(t)=t^pe^{t^2+\beta t}$ with $\beta>0$. Their results do not include the standard critical case $\lambda te^{t^2}$.

The first purpose of this paper is to solve the uniqueness problem for the elliptic equation with the standard critical exponential growth:

\begin{equation}\label{exp1}
\begin{cases}
-\Delta u =\lambda ue^{u^2},\quad\quad & x\in B_1,\\
u>0,\quad\quad & x\in B_1,\\
u=0,\quad\quad &x\in \partial B_1.
\end{cases}
\end{equation}

Adimurthi in \cite{Adimurthi0} proved that the above equation \eqref{exp1}  has a positive solution if and only if $\lambda\in (0, \lambda_1(B_1))$. By the method of moving-plane, we know that every positive solution of equation \eqref{exp1} must be radially decreasing and hence satisfies
\begin{equation}\label{adeq1}
		\begin{cases}
			-(ru')'=r\lambda ue^{u^2}, \quad\quad&  r\in(0,1),  \\
			u>0,\quad\quad&  r\in(0,1),\\
			u'(0)=0, u(1)=0.
		\end{cases}
	\end{equation}
Our first main result can be stated as follows:
\begin{theorem}\label{thm1}
For any $\lambda>0$, problem \eqref{adeq1} admits at most one solution.
\end{theorem}
\begin{remark}
Classical approaches   based on the non-degeneracy method has been successfully applied to solve the uniqueness problem if $f(t)$ has the subcritical or critical polynomial growth.
However, this method fails to deal with the exponential nonlinearity $f(t)$ like $te^{t^{\mu}}$, it is mainly because in this case we have $\lim\limits_{t\rightarrow +\infty}\frac{f'(t)}{f(t)}=+\infty$, which is significantly different from that in the case  of polynomial nonlinearity. To handle the critical exponential case, we will establish an elementary, but deep and powerful  result (see Lemma \ref{Lem3.1}). More precisely, we can show that there exists some $0<t<1$ such that $w=u-tv$ satisfies  $w>0$, $w'<0 \text{ in } (0,1)$ and $w(1)=w'(1)=0$, provided there exist two solutions $u$ and $v$.  Then we can deduce a contradiction through the local Pohozaev-type variational identity and a careful asymptotic expansion of $u$ and $v$ at the boundary.
\end{remark}

It was shown by Carleson-Chang
\cite{Carleson} that if $\alpha=4\pi$ and $\Omega$ is a ball, then the supremum in \eqref{Tru} can be achieved by a radial function $u_0$ satisfying
the equation \begin{equation*}
\begin{cases}
-\Delta u =\lambda_0 ue^{u^2},\quad\quad & x\in B_1,\\
u>0, & x\in B_1,\ \\
u=0,\quad\quad &x\in \partial B_1,
\end{cases}
\end{equation*}
for some $0<\lambda_0<\lambda_1(B_1)$. Hence, uniqueness result for ODE equation \eqref{adeq1} will be an important step towards solving the following uniqueness conjecture about the maximizers of the Trudinger-Moser inequality.

\begin{conjecture}
If $\alpha=4\pi$ and $\Omega$ is a ball, then the maximizers of the Trudinger-Moser inequality \eqref{Tru} are unique.

\end{conjecture}

From \cite{Adimurthi0}, we know that equation \eqref{exp1} admits a positive ground-state solution $u_{\lambda}$ with its functional energy $$I_\lambda(u)=\frac{1}{2}\int_{B_1}|\nabla u|^2dx-\frac{\lambda}{2}\int_{B_1}e^{u^2}dx<2\pi,$$ for any $0<\lambda<\lambda_1(B_1)$. In \cite{dMR},
del Pino, Musso and Ruf proved that if $\Omega$ in problem \eqref{exp1} is not a simply connected domain, then one can construct a family of positive solutions $u_\lambda$ satisfying the equation
\begin{equation}\label{adeq2}
\begin{cases}
-\Delta u =\lambda ue^{u^2},\quad\quad & x\in \Omega, \\
u>0,\quad\quad & x\in \Omega,\\
0<\lambda<\lambda_1(\Omega),\\
u=0,\quad\quad &x\in \partial \Omega,
\end{cases}
\end{equation}
such that $$J_\lambda(u_\lambda)=\frac{1}{2}\int_{\Omega}|\nabla u_\lambda|^2dx-\frac{\lambda}{2}\int_{\Omega}e^{u_\lambda^2}dx\rightarrow4\pi,$$ as $\lambda\rightarrow 0$. From their proof, one may conjecture if the domain of problem \eqref{exp1} is a simply connected or even a convex domain, there is no positive solution $u_\lambda$ to \eqref{adeq2} such that the above property holds. When $\Omega$ is a ball, we can give a positive answer to this problem. Indeed, through our uniqueness result, we see that each positive solution of \eqref{exp1} must be a ground-state solution. Consequently, this implies that one cannot construct a family of solutions $u_\lambda$ such that its functional energy exceeds $2\pi$.
More precisely, it can be stated as follows:

 \begin{corollary}
 Given any   family of positive solutions $u_\lambda$ satisfying the equation
\begin{equation}\label{ball}
\begin{cases}
-\Delta u =\lambda ue^{u^2},\quad\quad & x\in B_1, \\
u>0,\quad\quad & x\in B_1,\\
0<\lambda<\lambda_1(B_1),\\
u=0,\quad\quad &x\in \partial B_1,
\end{cases}
\end{equation}
there must hold $$I_\lambda(u_\lambda)=\frac{1}{2}\int_{B_1}|\nabla u_\lambda|^2dx-\frac{\lambda}{2}\int_{B_1}e^{u_\lambda^2}dx<2\pi.$$
\end{corollary}

\medskip

Slightly modifying the proof of Theorem \ref{thm1}, we can also obtain the uniqueness of positive solutions for elliptic equations with the critical exponential growth in the ball of hyperbolic space. Using this uniqueness result, we can establish the quantization property of positive solutions for corresponding elliptic equations. Quantization property for elliptic equations with the critical exponential growth  can be dated back to the work of Druet in \cite{D}, which can be stated as follows:
\vskip0.1cm

Let $\{u_k\}_k$ be a sequence of positive solutions of problem \eqref{adeq2} with $\lambda$ replaced by $\lambda_k$. Suppose that $u_k$ is bounded in $W^{1,2}_0(\Omega)$ and $u_k$ blows up. Then after passing to a subsequence, one has $\lambda_k\rightarrow \lambda_0$, $u_k\rightharpoonup u_0$ and there exists some integer $N$ such that
$$\lim\limits_{k\rightarrow +\infty}\|\nabla u_k\|_2^2= \|\nabla u_0\|_2^2+4\pi N.$$
When $\Omega$ is a disk, the positive solution $u_k$ must be radially symmetric through standard moving-plane method. (see e.g. \cite{ChenLi}, \cite{GNN}.) By using the ODE technique, Malchiodi and Martinazzi \cite{MM} proved that $u_0=0$ and $N=1$. Furthermore, they showed that the functional $I_\lambda$ under the constraint $\int_{B_1}|\nabla u|^2dx=\gamma$ does not admit any positive critical point for $\gamma$ sufficiently large. More recently, Druet and Thizy \cite{DT} showed in more general domain $\Omega$ that $u_0=0$ and $N$ is equal to the number of concentration points by the complicated blow-up analysis technique combining with a comparison theorem.
\medskip

In this paper, we will also utilize the uniqueness result to
develop a new strategy to establish the quantization property for elliptic equations with the critical exponential growth in the hyperbolic space. Our second main result reads as follows:

\begin{theorem}\label{adthm2}
Assume that $u_\lambda$ is a family of solutions satisfying
\begin{equation}\label{adeq3}
\begin{cases}
-\Delta_{\mathbb{H}} u =\lambda u e^{u^2},\quad\quad & x\in B_\mathbb{H}(0,R), \\
u>0,\quad\quad & x\in B_\mathbb{H}(0,R),\\
0<\lambda<\lambda_1(B_\mathbb{H}(0,R)),\\
u=0,\quad\quad &x\in \partial B_\mathbb{H}(0,R),
\end{cases}
\end{equation}
where $\mathbb{H}$ denotes the standard hyperbolic space $\mathbb{H}=(B_1, dV_\mathbb{H})$, $dV_\mathbb{H}=(\frac{2}{1-|x|^2}))^2dx$, $B_\mathbb{H}(0,R)$ denotes the ball in hyperbolic space centered at the origin with the geodesic radius $R$, $-\Delta_{\mathbb{H}}$ denotes Laplace-Beltrami operator in $\mathbb{H}$, $\lambda_1(B_\mathbb{H}(0,R))$ is the first eigenvalue of the operator $-\Delta_{\mathbb{H}}$ with the Dirichlet boundary in $B_\mathbb{H}(0,R)$. Then $u_\lambda$ is radially symmetric and unique. Furthermore, when $\lambda\rightarrow \lambda_0$, we have
\medskip

(i) if $\lambda_0=0$, then $u_\lambda$ blows up at the origin, and $|\nabla_{\mathbb{H}}u_\lambda|^2dV_{\mathbb{H}}\rightharpoonup 4\pi \delta_0$, $\lambda u_\lambda e^{u_\lambda^2}\rightharpoonup 4\pi \delta_0$,
\medskip

(ii) if $\lambda_0\in (0,\lambda_1(B_\mathbb{H}(0,R)))$, then $u_\lambda\rightarrow u_0$ in $C^{2}(B_\mathbb{H}(0,R)))$ and $u_0$ is a positive radial solution of the equation
\begin{equation}
\begin{cases}
-\Delta_{\mathbb{H}} u=\lambda_0 ue^{u^2},\quad\quad & x\in B_\mathbb{H}(0,R), \\
u>0,\quad\quad & x\in B_\mathbb{H}(0,R), \\
u=0,\quad\quad &x\in \partial B_\mathbb{H}(0,R).
\end{cases}
\end{equation}
\vskip 0.1cm

(iii) if $\lambda_0=\lambda_1(B_\mathbb{H}(0,R))$, then $u_\lambda\rightarrow 0$ in $C^{2}(B_\mathbb{H}(0,R))$.
\end{theorem}

\begin{remark}
The existence of ground state solution (consequently positive solution) for elliptic equation \eqref{adthm2} can be verified by following the same line in the  works, e.g. in \cite{Adimurthi0, CLZ1, CLZ2}.
\end{remark}

\begin{remark}
Our proofs of quantization result on the hyperbolic spaces are based on the uniqueness result of solutions to the corresponding equation \eqref{adeq3}. Once the uniqueness property is obtained, we can study the quantization properties of the least energy solutions instead of the positive solutions. Our method is very simple and we can avoid using the complicated blow-up analysis of ODEs. We stress that this method can be also applied to study the quantization result for high dimensional ball of $\mathbb{R}^n$ or general hyperbolic space $\mathbb{H}^n$, provided the uniqueness result for solutions to the corresponding equations is established. This will be carried out in our forthcoming work.
\end{remark}

\begin{remark}
Since the uniqueness result for problem \eqref{adeq3} holds for any fixed $\lambda$, hence we need not to choose a subsequence of $\{u_\lambda\}_{\lambda}$ to obtain quantization result. Furthermore, in our proofs for the quantization result, we get rid of the assumption that $u_\lambda$ is uniformly bounded in $W^{1,2}(B_\mathbb{H}(0,R))$, which was required in \cite{D, DT}, hence we can directly obtain the  non-existence and multiplicity of positive critical point for supercritical Trudinger-Moser functional in $B_\mathbb{H}(0,R)$.
\end{remark}

\begin{theorem}\label{adthm3}
There exists $\gamma^{*}>4\pi$ such that the
Trudinger-Moser functional $F(u)=\int_{B_\mathbb{H}(0,R)}(e^{u^2}-1)dV_{\mathbb{H}}$ under the constraint $\int_{B_\mathbb{H}(0,R)}|\nabla_{\mathbb{H}}u_\lambda|^2dV_{\mathbb{H}}=\gamma$ has at least two positive critical point for $\gamma \in (4\pi, \gamma^{*})$, at least one critical point for $\gamma=4\pi$ or $\gamma=\gamma^{*}$, no positive critical point for  $\gamma \in (\gamma^{*}, +\infty)$.
\end{theorem}

\begin{remark}
Exploiting the existence of the critical points of  the Trudinger-Moser functional $M(u)=\int_{\Omega}(e^{u^2}-1)dx$ under the constraint
$\int_{\Omega}|\nabla u|^2dx=\sigma$ for $\sigma>4\pi$ has been a challenging problem. By using a variational method and the monotonicity of the functional $M(u)$, Struwe \cite{ST} proved that there exits $\sigma^{*}>4\pi$ such that $M(u)$ has at least two positive critical points for almost $\gamma\in (4\pi,\sigma^{*})$. Later, Lamm, Robert and Struwe \cite{LRS} introduced the Trudinger-Moser flow and strengthened it to every $\gamma\in (4\pi,\sigma^{*})$. When $\Omega$ is a disk, Malchiodi and Martinazzi \cite{MM} applied refined blow-up analysis for radial critical point of $M$ to derive that $M(u)$ does not admit any positive critical point for $\sigma$ sufficiently large.
It is conjectured that if $\Omega$ is a simply connected domain, the above non-existence result still holds.
\end{remark}

Finally, we note that the following improved Trudinger-Moser inequality (see \cite{T}) still holds:
\begin{equation}\label{egen}\sup_{u\in W^{1,2}_0(B_1),\int_{B_1}(|\nabla u|^2-\lambda |u|^2)dx\leq 1} \int_{B_1}e^{4\pi u^2}dx<+\infty \end{equation}
if $\lambda<\lambda_1$, where $\lambda_1$ denotes the first eigenvalue of operator $-\Delta$ with the Dirichlet
boundary condition. Furthermore, it is proved in \cite{Yang}  that the improved Trudinger-Moser inequality \eqref{egen} admits an extremal function through the blow-up technique. Hence it is also interesting to consider the uniqueness for the extremal of \eqref{egen} and the problems of positive critical point for super-critical Trudinger-Moser functional $H(u)=\int_{B_1}(e^{u^2}-1)dx$ under the Dirichlet energy constraints $\int_{B_1}(|\nabla u|^2-\lambda |u|^2)dx=\beta$ for $\beta>4\pi$. In fact, slightly modifying the proofs of our Theorems \ref{thm1}, \ref{adthm2} and \ref{adthm3}, we can also obtain the uniqueness of the problem \eqref{adeq3}, the quantization results, multiplicity and non-existence for positive critical point of supercritical Trudinger-Moser functional. We only list these results without giving the detailed proof.
\begin{theorem}\label{thm5}
For any $\lambda< \lambda_1(B_1)=2\pi$. Let $u_\theta$ be a positive solutions of equation
\begin{equation}\label{quantization}
\begin{cases}
-\Delta u-\lambda u =\theta u e^{u^2},\quad\quad & x\in B_1, \\
u>0,\quad\quad & x\in B_1, \\
u=0,\quad\quad &x\in \partial B_1,\\
\end{cases}
\end{equation}
where $\theta$ satisfies $0<\theta<\lambda_1(B_1)-\lambda$. Then $u_\theta$ is radially decreasing and unique. Furthermore, when $\theta\rightarrow \theta_0$, there holds
\vskip0.1cm

(i) If $\theta_0=0$, then $u_\theta$ blows up at the origin, $\lambda u_\theta^2\exp(u_\theta^2)dx\rightharpoonup 4\pi \delta_0$ and $\|\nabla u_\theta\|_2^2\rightarrow 4\pi$ as $\theta\rightarrow0$.
\medskip

(ii) if $\theta_0\in (0,\lambda_1(B_1)-\lambda)$, then $u_\lambda\rightarrow u_0$ in $C^{2}(\bar{B_1})$ and $u_0$ is a positive radial solution of the equation
\begin{equation}
\begin{cases}
-\Delta u_0-\lambda u_0 =\theta_0 u_0e^{u_0^2}\quad\quad & x\in B_1, \\
u_0>0,\quad\quad & x\in B_1, \\
u_0=0, \quad\quad &x\in \partial B_1.
\end{cases}
\end{equation}
\medskip

(iii) if $\theta_0=\lambda_1(B_1)-\lambda$, then $u_\theta\rightarrow 0$ in $C^{2}(\bar{B_1})$.
\end{theorem}



\begin{theorem}\label{thm6}
There exists $\beta^{*}>4\pi$ such that the Trudinger-Moser functional $H(u)=\int_{B_1}\big(e^{u^2}-1\big)dx$ under the  constraints $\|\nabla u\|_2^2-\lambda\|u\|_2^2=\beta$ does not admit any positive critical point for $\beta>\beta^{*}$ and admit at least two positive critical points for $\beta\in (4\pi, \beta^{*})$.
\end{theorem}
\begin{remark}
When $\lambda=0$, the multiplicity and non-existence of positive critical point of super-critical Trudinger-Moser functional have been obtained by Malchiodi and Martinazzi in \cite{MM} through the accurate Dirichlet energy expansion formula obtained by the ODE technique.  It should be noted that even if $\lambda=0$, there is not any upper bound estimate for $\beta^{*}$. A famous conjecture proposed by Malchiodi is that $\beta^{*}$ should be equal to $8\pi$. Using the method  we developed in this paper, we can deduce that $\beta^{*}<8\pi$ if we consider the perturbed Trudinger-Moser functional $\tilde{H}(u)=\int_{B_1}(e^{u^2}-1-u^2)dx$ under the constraint $\int_{B_1}|\nabla u|^2dx=\beta$.
\end{remark}

\begin{theorem}\label{thm5}The
perturbed supercritical Trudinger-Moser functional $\tilde{H}(u)=\int_{B_1}(e^{u^2}-1-u^2)dx$ under the constraint $\int_{B_1}|\nabla u|^2dx=\beta$ does not admit any positive critical point if $\beta\geq 8\pi$.
\end{theorem}

\section{The proof of Theorem \ref{thm1}}
In order to prove the uniqueness of positive solutions of  equation \eqref{adeq1}, we need some lemmas.
\begin{lemma}\label{Lem3.2}
Let $u$ and $v$ be non-negative $C^1$ functions on $[R_1,R_2]$, with $v\neq 0$ and $u>0$ in $(R_1,R_2)$. Assume that $v(R_i)=0$ and $u'(R_i)\ne 0$ if $u(R_i)=0$, where $i=1,2$. Then there exists a unique  $t\in(0,\infty)$ such that $w=u-tv\geq0$ in $[R_1,R_2]$ and $w(\xi)=0$ for some $\xi\in[R_1,R_2]$. Furthermore

(i) if $\xi\in(R_1,R_2)$ and $w$ is in $C^2$ in a neighbourhood of $\xi$, then $w(\xi)=w'(\xi)=0$, $w''(\xi)\geq 0$;\\

(ii) if $\xi=R_i$ and $u(R_i)=0$ for some $i\in\{1,2\}$, then $w'(\xi)=0$;\\

(iii) if $u(R_i)\ne 0$ and $v(R_i)=0$, then $\xi\ne R_i$, where $i=1,2$.
\end{lemma}
\begin{proof}This result was first essentially established by Rabinowitz \cite{Rabinowitz} (see also \cite{Adimurthi}). Since the proof is short, for the convenience of the reader, we prefer to give the proof.
 By the hypotheses, $v/u$ is a continuous function on $[R_1,R_2]$. Let $w=u-tv$, where
 \begin{equation*}\begin{split}
t=\max\bigg\{s:\frac{sv}{u}\leq1\bigg\}=\bigg[\max\frac{v}{u}\bigg]^{-1}.
\end{split}
\end{equation*}
Then $w\geq 0$ and there exists a $\xi\in[R_1,R_2]$ such that $w(\xi)=0$. If $\xi\in(R_i,R_2)$, then $\xi$ is a local minimizer and hence (i) holds. If $\xi=R_i$  for some $i\in\{1,2\}$, then by 1'H\^{o}pital's rule
	\begin{equation*}
		\begin{split}
			\frac{1}{t}=\lim_{r\to R_i}\frac{v(r)}{u(r)}=\frac{v'(R_i)}{u'(R_i)}
		\end{split}
	\end{equation*}
	and hence $w'(\xi)=0$. This proves (ii).
We prove (iii) by contradiction. If $\xi=R_i$, then we have $\frac{1}{t}=\frac{v(R_i)}{u(R_i)}=0$, this  impossible since $t\in(0,\infty)$. Hence (iii) holds.

    \end{proof}
  \begin{lemma}\label{Lem3.3}
	Let $u$ and $v$ be two solutions of problem \eqref{adeq1} and $r_1, r_2\in(0,\infty)$. For any $t>0$, set  $w=u-tv$. Then we have

(i) if $w(r_1)=w'(r_1)=0$ and $w''(r_1)\geq 0$, then $t\leq1$.\\

(ii) if $t=u(0)/v(0)$ and $w>0$ in $(0,r_2)$, then $t\leq1$.\\

In either case, $t=1$ implies that $u\equiv v$.
\end{lemma}
\begin{proof}
If $w(r_1)=w'(r_1)=0$, then $u(r_1)=tv(r_1)$. Combining this, \eqref{adeq1} and the assumption of (i), we have
\begin{equation}\label{eqin}
0\geq-\bigg(w''+\frac{1}{r}w'\bigg)(r_1)=\lambda tv(r_1)(e^{u^2(r_1)}-e^{v^2(r_1)}).\end{equation}
Hence  $e^{t^2v^2(r_1)}=e^{u^2(r_1)}\leq e^{v^2(r_1)}$, which implies that $t\leq1$. This proves (i).
\medskip

Now we prove (ii) by contradiction. Suppose that (ii) is not true. From the assumption that $w>0$ in $(0,r_2)$, we observe that for any $0<\eta<r_2$, there exists a $\xi\in(0,\eta)$ such that $w'(\xi)>0$. Now, we claim that there exists $\varepsilon_0>0$ such that $w'\geq0$ in $[0,\varepsilon_0]$. Indeed, if this is not true, then there exists a sequence $\xi_k\to0$ such that $w'(\xi_k)<0$. This implies that $w'$ changes sign infinitely often near zero and then we can find a sequence $\eta_k\to 0$ such that $w'(\eta_k)\geq0$, $w''(\eta_k)\geq0$. Similar as (\ref{eqin}), from the assumption that $w>0$ in $(0,r_2)$ and \eqref{adeq1}, we have
\begin{equation}\label{eqre}
0\geq-\bigg(w''+\frac{1}{r}w'\bigg)(\eta_k)>\lambda tv(\eta_k)(e^{u^2(\eta_k)}-e^{v^2(\eta_k)}).\end{equation}
Letting $\eta_k\to0$, we obtain $e^{u^2(0)}\leq e^{v^2(0)}$, and hence $t=u(0)/v(0)\leq 1$. This contradict with $t>1$, and the claim is proved.

 Since $t>1$, it follows from the above claim that $w'\geq0$ in $[0,\varepsilon_0]$ and $w'(0)=0$. Hence there exists a sequence $\eta_k\to 0$ such that $w''(\eta_k)\geq0$. Now by repeating the same argument at $\eta_k$ as (\ref{eqre}), we conclude that $t\leq 1$, which is a contradiction. This proves (ii).
\medskip

Finally, if $t=1$, then $u(p)= v(p)$ and $u'(p)= v'(p)$, where $p\in\{r_1,0\}$, which implies $u\equiv v$.
\end{proof}

\begin{lemma}\label{Lem3.1}
	     Let $u\neq v$ be two solutions of problem \eqref{adeq1}. Then there exist $t\in(0,1)$, $\xi\in(0,1)$ such that for $w=u-tv$ one of the following holds:\\

\text{(a)} $u(0)=tv(0)$,\\

\text{(b)} $w\geq0,w'\leq0\text{ in }[0,\xi],w(\xi)=w'(\xi)=0$,\\

\text{(c) } $w\geq0,w'\leq0\text{ in }[0,1],w'(1)=0$.
\end{lemma}
\begin{proof}
	We first claim that there exist $t\in(0,1)$ and $\xi\in(0,1)$ such that for $w=u-tv$, one of the following conditions holds:

(I) $u(0)=tv(0)$,\\

(II) $w\geq0\text{ in }[0,\xi],w(\xi)=w'(\xi)=0,w''(\xi)\geq0$,\\

(III)$w\geq0,w'\leq0\text{ in }[0,1],w'(1)=0$.
\medskip

From Lemma \ref{Lem3.2}, we can choose $t_1>0$, $\xi_1\in[0,1]$ such that $w_1=u-t_1v$ satisfies $w_1\geq0$ in $[0,1]$, $w(\xi_1)=w'(\xi_1)=0$. Indeed, we have $t_1<1$, since if $t_1\geq1$, it holds $v\leq t_1v\leq u$, this implies that $u\equiv v$, which is a contradiction.

If $\xi_1=0$, then $u(0)=t_1v(0)$ and (i) holds. If $\xi_1\in(0,1)$, then $\xi_1$ is a local minimum and (ii) holds.

Again from Lemma \ref{Lem3.2}, for every $\eta\in(0,\infty)$, one can choose $t_\eta>0$, $\xi_\eta\in[0,\eta]$, such that $w_\eta=u-t_\eta v$ satisfies $w_\eta\geq0$ in $[0,\eta]$, $w_\eta(\xi_\eta)=0$. Now claim that if (I) and (II) do not hold, then $\xi_\eta=\eta$. We prove this by contradiction. Indeed, if $\xi_\eta\in(0,\eta)$, then $\xi_\eta$ is a local minimizer for $w_\eta$ and therefore $w_\eta'(\xi_\eta)=0$, $w_\eta''(\xi_\eta)\geq0$. Hence from (i) of Lemma \ref{Lem3.3} we get $t_\eta<1$, which contradicts the assumption that (II) does not hold. If $\xi_\eta=0$, from (ii) of Lemma \ref{Lem3.3}, we conclude that $t_\eta<1$. This contradicts the assumption that (I) does not hold.

Now, we show that if (I) and (II) do not hold, then $w'(r)\leq0$ for $r\in[0,1]$. We prove this by contradiction. Suppose that there exists $\eta_0\in(0,1)$ such that $w'_1(\eta_0)>0$. We consider the interval $[0,\eta_0]$. From the above argument, we see that $\xi_{\eta_0}=\eta_0$, thus we have $t_{\eta_0}=u(\eta_0)/v(\eta_0)$. Hence for $0\leq r<\eta_0\leq \infty$, we deduce from $w_{\eta_0}(r)\geq0$ that
	$$\frac{u(r)}{v(r)}\geq \frac{u(\eta_0)}{v(\eta_0)}.$$
	Thus $u/v$ is a decreasing function, this implies that $u/v\leq u'/v'$. Let $1>r>\eta_0$, so that
	$$\frac{u(r)}{v(r)}\leq \frac{u(\eta_0)}{v(\eta_0)}\leq \frac{u'(\eta_0)}{v'(\eta_0)}<t_1$$
	and hence $w_1(r)<0$, which is a contradiction. Therefore, $w'_1\leq 0$ in $[0,1]$ and this proves (III), and the claim is proved.
	\medskip

	Next, we show that if (I) and (III) do not hold, then (b) holds. Let
	$$\eta_0=\inf\{\eta>0;\text{  (II) holds in } [0,\eta]\}.$$
	We claim that $\eta_0>0$. We prove this by contradiction. Suppose that $\eta_0=0$, then there exists a sequence $\eta_k\to0$ with $t_{\eta_k}\to t_0$ for some $t_0\leq1$ such that $w_{\eta_k}(\eta_k)=0$. This implies that $u(0)=t_0v(0)$. If $t_0=1$,  then $u\equiv v$, which is a contradiction. If $t_0<1$, then (I) holds, which contradicts our assumption, hence $\eta_0>0$. Now by the definition of $\eta_0$, we have $w_{\eta_0}(\eta_0)=w'_{\eta_0}(\eta_0)=0$, $w''_{\eta_0}(\eta_0)\geq0$. Then it follows from (I) of Lemma \ref{Lem3.3} that $t_{\eta_0}<1$.
\vskip 0.1cm

Now, we show that  $w'_{\eta_0}\leq0$ in $[0,\eta_0]$. We assert that for any  $0<\eta<\eta_0$, one has $\xi_\eta=\eta$, $t_\eta=u(\eta)/v(\eta)$. Indeed, if $\xi_\eta\in(0,\eta)$, then $\xi_\eta$ is a local minimizer of $w_\eta$, and hence $w'_\eta(\xi_\eta)=0$, $w''_\eta(\xi_\eta)\geq0$. Therefore from (I) of Lemma \ref{Lem3.3}, $t_\eta<1$, but this is impossible from the definition of $\eta_0$. If $\xi_\eta=0$, then $t_\eta=u(0)/v(0)$ and hence from (II) of Lemma \ref{Lem3.3}, $t_\eta<1$. Thus (I) holds, which contradicts our assumption. Therefore we have $\xi_\eta=\eta$ and $t_\eta=u(\eta)/v(\eta)$. This implies that $w_\eta(r)>0$ for $0<r<\eta$, and hence $u(r)/v(r)>u(\eta)/v(\eta)$. Thus $u/v$ is a decreasing function in $[0,\eta_0]$ and $u/v\leq u'/v'$ in $[0,\eta_0]$. Suppose that there exists $\eta_1<\eta_0$ such that $w'_{\eta_0}(\eta_1)>0$, then for $\eta_1<r<\eta_0$ we have
$$\frac{u(r)}{v(r)}\leq \frac{u(\eta_1)}{v(\eta_1)}\leq \frac{u'(\eta_1)}{v'(\eta_1)}<t_{\eta_0}.$$
Hence $w_{\eta_0}(r)<0$, which is a contradiction. Therefore $w'_{\eta_0}\leq0$ in $[0,\eta_0]$, and (b) holds in $[0,\eta_0]$. The proof of this lemma is finished.

\end{proof}

\begin{lemma}\label{Lem2.1}
	Let $u$ and $u_2$ be two solutions of \eqref{adeq1} with $u(0)<u_2(0)$. Then there exists a solution $u_1$ of \eqref{adeq1} such that $u_1(0)<u_2(0)$ and $u_1$ and $u_2$ intersect at most once in $(0,1)$.
\end{lemma}
\begin{proof}

Let $\alpha>0$ and $w(\cdot,\alpha)$ be the unique solution of the following initial-value problem,
\begin{equation}
	\begin{cases}
		-(rw')'=rf(w), \\
		w(0)=\alpha ,  w'(0)=0,
	\end{cases}
\end{equation}
where $f(w)=\lambda we^{w^2}$. We also denote $R(\alpha)$ the first zero of $w(\cdot,\alpha)$ defined by
$$R(\alpha)=\sup\{r:w(s,\alpha)>0 \text{ for all }  s\in[0,r]\}.$$
Assume that $u$ and $u_2$ intersect at least twice (otherwise there is nothing to prove). Let $0 < R_1 < R_2 < 1$ be the first two consecutive points of intersection with $u(r)<u_2(r)$ for all $r\in(0,R_1)$.
Clearly, $w(r,\alpha_0)=u(r)$ and $w(r,\alpha_2)=u_2(r)$, where $\alpha_0=u(0)$, $\alpha_2=u_2(0)$ with $R(\alpha_0)=R(\alpha_2)=1$.
Let $\alpha<\alpha_0$ and be close to $\alpha_0$, and  $0 < R_1(\alpha) < R_2(\alpha) < \infty$ be the first two consecutive points of intersections of $w(\cdot,\alpha)$ with $u_2$ (which exist by continuity) such that $w(r,\alpha)\leq u_2(r)$ for $r\in(0,R_1(\alpha))$.

Now, as $\alpha$ moves towards zero, one of the following three possibilities holds.

(i) There exists a $\alpha_1\in(0,\alpha_0)$ such that $R_2(\alpha_1) =1$. Then the conclusion of the lemma holds.\\

(ii) There exists a $\alpha_1\in(0,\alpha_0)$ and a point $R\in(0,1)$ such that
$$w(R,\alpha_1)=u_2(R),\quad\quad\quad\quad w'(R,\alpha_1)=u_2'(R). $$
Then, by uniqueness of the initial-value problem, $w(r,\alpha_1) = u_2(r)$ for all $r \in(0,1)$, which is a contradiction.\\

(iii) $0 < R_1(\alpha) < R_2(\alpha) < 1$ for all $\alpha$ and $\lim_{\alpha\to0} (R_2(\alpha)- R_1(\alpha))=0$. In this cases, let $I(\alpha)=[R_1(\alpha) , R_2(\alpha)]$ and $v(r)=w(r,\alpha)-u_2(r)$. Then $v(r)$ satisfies
\begin{equation}\label{R2.3}
	\begin{cases}
		-(rv')'=Qv, \quad\quad & r\in I(\alpha), \\
		v>0, \quad\quad & r\in I(\alpha), \\
		v(R_1(\alpha))=v(R_2(\alpha))=0,
	\end{cases}
\end{equation}
where
$$Q(r)=r\frac{f(w(r,\alpha)-f(u_2(r))}{w(r,\alpha)-u_2(r)}.$$
Hence $v$ is the first eigenfunction with eigenvalue $\mu_1 = 1$ of the following eigenvalue problem:
\begin{equation*}
	\begin{cases}
		-(r\varphi')'=\mu Q\varphi,  \quad\quad & r\in I(\alpha), \\
		\varphi=0,  & r\in  \partial I(\alpha).
	\end{cases}
\end{equation*}
For $\alpha\in(0,\alpha_0)$ and $0<R_1\leq R_1(\alpha), $ from (H2) we derive
\begin{equation}
	M=\sup\{Q(r):  r\in I(\alpha),\alpha\in(0,\alpha_0)\}<\infty.
\end{equation}
Let $\lambda_1(\alpha)$ be the first eigenvalue of
\begin{equation}\label{R2.5}\begin{cases}
	-\frac{d^{2}\varphi}{dr^2}=\lambda\varphi, \quad &x\in I(\alpha), \notag\\
	\varphi=0, \quad\quad &x\in  \partial I(\alpha).
\end{cases}\end{equation}
Then, from \eqref{R2.3}-\eqref{R2.5}, we get
\begin{equation}\label{R2.6}
	\begin{split}
		1&=\inf\bigg\{\frac{\int_{I(\alpha)}{r(\varphi')^2}dr}{\int_{I(\alpha)}{Q\varphi^2}dr}; \varphi\in H_0^1(I(\alpha))\bigg\}\\
		&\geq\frac{R_1}{M}\bigg\{\frac{\int_{I(\alpha)}{(\varphi')^2}dr}{\int_{I(\alpha)}{\varphi^2}dr}; \varphi\in H_0^1(I(\alpha))\bigg\}\\
		&\geq\frac{R_1}{M}\lambda_1(\alpha).
	\end{split}
\end{equation}
Since $R_2(\alpha)- R_1(\alpha)\to0$ as $\alpha\to0$, we derive that $\lambda_1(\alpha)\to \infty$, which contradicts \eqref{R2.6}. This proves the lemma.\\
\end{proof}
 \begin{lemma}\label{lem2.5}
 Let $u$ and $u_2$ be two solutions of \eqref{adeq1} with $u(0)<u_2(0)$, then there exists $u_1$ of \eqref{adeq1} such that $u_1(0)<u_2(0)$ and $u_1$ and $u_2$ intersect only once in $(0,1)$. Moreover, $\frac{u_1(r)}{u_2(r)}$ is strictly increasing.
 \end{lemma}
 \begin{proof}
According to Lemma \ref{Lem2.1}, in order to prove that $u_1$ and $u_2$ intersect only once in $(0,1)$, we only need to prove that
$u_1$ and $u_2$ intersect at least once. We argue this by contradiction. Assume that $u_1$ and $u_2$ does not intersect in $(0,1)$, then $u_2>u_1$ in $(0, 1)$. Since $u_1$ and $u_2$ satisfies equation \eqref{adeq1}, we have
$$\int_{B_1}(-\Delta u_2 u_1+\Delta u_1 u_2)dx=\lambda\int_{B_1}u_2u_1(e^{u_2^2}-e^{u_1^2})dx>0.$$
On the other hand, integration by parts directly gives that
$$\int_{B_1}(-\Delta u_2 u_1+\Delta u_1 u_2)dx=2\pi(u_1'(1)u_2(1)-u_2'(1)u_1(1))=0,$$
which is obviously a contradiction.

Now, we start to prove that $\frac{u_1(r)}{u_2(r)}$ is strictly increasing. Assume that the only intersection point of $u_1$ and $u_2$ is $r_0$, then $u_2(r)\geq u_1(r)$ in $(0, r_0]$. By equation \eqref{adeq1}, one can directly obtain that for any $r\in (0,r_0]$,
\begin{equation*}\begin{split}
\int_{B_r}(-\Delta u_2 u_1+\Delta u_1 u_2)dx=\int_{B_r}u_2u_1(e^{u_2^2}-e^{u_1^2})dx>0.
\end{split}\end{equation*}
On the other hand, integration by parts directly gives
$$\int_{B_r}(-\Delta u_2 u_1+\Delta u_1 u_2)dx=2\pi(u_1'(r)u_2(r)-u_2'(r)u_1(r))r.$$
Combining the above estimate, we derive that $u_1'(r)u_2(r)-u_2'(r)u_1(r)>0$ for $r\in (0,r_0]$, that is $\frac{u_1(r)}{u_2(r)}$ is strictly increasing in $(0,r_0]$. Now, we will prove that $u_1'(r)u_2(r)-u_2'(r)u_1(r)>0$ holds for $r\in (r_0, 1)$. By equation \eqref{adeq1}, one can similarly obtain
$$\int_{B_1\setminus B_r}(-\Delta u_2 u_1+\Delta u_1 u_2)dx=\int_{B_1\setminus B_r}u_2u_1(e^{u_2^2}-e^{u_1^2})dx<0.$$
Again applying integration by parts, one can derive that
$$2\pi(u_1'(r)u_2(r)-u_2'(r)u_1(r))r=-\int_{B_1\setminus B_r}(-\Delta u_2 u_1+\Delta u_1 u_2)dx>0.$$
Then we accomplish the proof of Lemma \ref{lem2.5}.
\end{proof}

Now, we are in position to give the proof of the uniqueness of positive radial solution of equation \eqref{adeq1}.
\begin{proof}[Proof of Theorem \ref{thm1}]
Assume that $u$ and $v$ are two positive solutions of equation \eqref{adeq1} with $v(0)<u(0)$, $u$ and $v$ intersecting only once in $(0,1)$. According to Lemma \ref{Lem3.1}, there exist $t\in(0,1)$, $\xi\in(0,1)$ such that for $w=u-tv$ one of the following holds:\\

\text{(a)} $u(0)=tv(0)$,\\

\text{(b)} $w\geq0,w'\leq0\text{ in }[0,\xi],w(\xi)=w'(\xi)=0$,\\

\text{(c) } $w\geq0,w'\leq0\text{ in }[0,1],w'(1)=0$.
\medskip

Obviously, $(a)$ does not occur. According to Lemma \ref{lem2.5}, we know that $\frac{v(r)}{u(r)}$ is strictly increasing in $(0,1)$. This gives that for any $r\in (0,1)$, $\frac{v'(r)}{u'(r)}<\frac{v(r)}{u(r)}$. Hence it follows that there does not exists $\xi \in (0,1)$ such that $w(\xi)=w'(\xi)=0$, that is to say that $(b)$ is also impossible. In order to prove the uniqueness theorem, we only need to prove that $(c)$ does not happen.  Assume that $(c)$ happens, then $w(1)=w'(1)=0$. A simple calculation combining with $u'(1)=tv'(1)$ and $u(1)=v(1)=0$ yields that
\begin{equation}\label{r1}
u^{(2)}(1)=tv^{(2)}(1),\ u^{(3)}(1)=tv^{(3)}(1),\ u^{(4)}(1)=tv^{(4)}(1)
\end{equation}
and
\begin{equation}\label{r2}
u^{(5)}(1)-tv^{(5)}(1)=-6\big((u')^3(1)-t(v')^3(1)\big)=-6u'(1)\big((u')^2(1)-(v')^2(1)\big).
\end{equation}

Furthermore, we also have
\begin{equation*}\begin{split}\label{t2}
u^{(6)}(1)-tv^{(6)}(1)+u^{(5)}(1)-tv^{(5)}(1)+36(u')^2(1)u''(1)-36t(v')^2(1)v''(1)=0.
\end{split}\end{equation*}
This together with $u''(1)+u'(1)=0$ and \eqref{r2} gives that
\begin{equation}\begin{split}\label{t2}
u^{(6)}(1)-tv^{(6)}(1)&=36(u')^3(1)-36t(v')^3(1)-\big(u^{(5)}(1)-tv^{(5)}(1)\big)\\
&=36u'(1)\big((u'(1))^2-(v'(1))^2)\big)+6u'(1)\big((u'(1))^2-(v'(1))^2\big)\\
&=42u'(1)\big((u'(1))^2-(v'(1))^2\big)
\end{split}\end{equation}

Since $u$ and $v$ satisfy equation \eqref{adeq2}, using the Pohozaev identity and combining the radial symmetry of $u$ and $v$, we have
$$\int_{B_r}-\Delta u(x\cdot\nabla u)dx=-\frac{1}{2}\int_{\partial B_r}|\nabla u|^2(x\cdot \nu)d\sigma=-\pi r^2(u')^2(r)$$
and
\begin{equation*}\begin{split}
\int_{B_r}ue^{u^2}(x\cdot\nabla u)=&\frac{1}{4}\int_{B_r}\nabla(e^{u^2}-1)\nabla(|x|^2)\\
&=-\frac{1}{4}\int_{B_r}(e^{u^2}-1)\Delta(|x|^2)dx+\frac{1}{4}\int_{\partial B_r}(e^{u^2}-1)\frac{\partial |x|^2}{\partial \nu}ds\\
&=-\int_{B_r}(e^{u^2}-1)dx+\pi r^2(e^{u^2(r)}-1).
\end{split}\end{equation*}
That is $$\pi r^2(u')^2(r)=\int_{B_r}(e^{u^2}-1)dx-\pi r^2(e^{u^2(r)}-1).$$
This deduces that
\begin{equation}\label{xin1}
\pi r^2(u')^2(1)-\pi r^2(u')^2(r)=\int_{B_1\setminus B_r}(e^{u^2}-1)dx+\pi r^2(e^{u^2(r)}-1).
\end{equation}

Similarly, we also have
\begin{equation}\label{xin2}
\pi r^2(v')^2(1)-\pi r^2(v')^2(r)=\int_{B_1\setminus B_r}(e^{v^2}-1)dx+\pi r^2(e^{v^2(r)}-1).
\end{equation}

Multiplying \eqref{xin2} by $t^2$, and  then subtracting  \eqref{xin1}, we can obtain from  $u'(1)=tv'(1)$ that
\begin{equation}
\label{ad21}\begin{split}
\pi r^2((u')^2(r)-t^2(v')^2(r))&=\int_{B_1\setminus B_r}\big(t^2(e^{v^2}-1)-(e^{u^2}-1)\big)dx+\pi r^2\big(t^2(e^{v^2}-1)-(e^{u^2}-1)\big)\\
&=\pi r^2(I+II).
\end{split}\end{equation}
In view of \eqref{r1},\eqref{r2} and \eqref{t2}, we have
\begin{equation}\label{r3}\begin{split}
&(u')^2(r)-t^2(v')^2(r)=(u'(r)+tv'(r))(u'(r)-tv'(r))\\
=&\big(2u'(1)+2u''(1)(r-1)+O((r-1)^2)\big) \times \big(\frac{1}{4!}(u^{(5)}(1)-tv^{(5)}(1))(r-1)^4+\\ &+\frac{1}{5!}(u^{(6)}(1)-tv^{(6)}(1)(r-1)^5+O((r-1)^6)\big)\\
=&-\frac{1}{2}(u')^2(1)\big((u')^2(1)-(v')^2(1)\big)(r-1)^4+\frac{1}{2}(u')^2(1)\big((u')^2(1)-(v')^2(1)\big)(r-1)^5+\\
&\ \ +\frac{84}{120}(u')^2(1)\big((u')^2(1)-(v')^2(1)\big)(r-1)^5+O((r-1)^6).
\end{split}\end{equation}
For $II$, we have
\begin{equation*}\begin{split}
II&=\big(t^2v^2(r)-u^2(r)+\frac{1}{2}t^2v^4(r)-\frac{1}{2}u^4(r)\big)+O((r-1)^6)\\
&=\big((tv(r)+u(r))(tv(r)-u(r))+\frac{1}{2}\big((tv)^2v^2(r)-u^4(r)\big)\big)+O((r-1)^6)\\
&=\big(2u'(1)(r-1)+O((r-1)^2)\big)\big((\frac{tv^{(5)}(1)-u^{(5)}(1)}{5!})(r-1)^5+O((r-1)^6)\big)+\\
&\ \ +\frac{1}{2}\big(u'(1)(r-1)+\frac{1}{2}u''(1)(r-1)^2+O((r-1)^3)\big)^2\big(v'(1)(r-1)+\frac{ v''(1)}{2}(r-1)^2+\\ &\ \ + O((r-1)^3)\big)^2-\frac{1}{2}\big(u'(1)(r-1)+\frac{1}{2}u''(1)(r-1)^2+O((r-1)^3)\big)^4+O((r-1)^6)\\
&=\frac{1}{2}\big(u'(1)(r-1)+\frac{1}{2}u''(1)(r-1)^2+O((r-1)^3)\big)^2\times \Big(\big(v'(1)(r-1)+\frac{1}{2}v''(1)(r-1)^2+\\&\ \ +O((r-1)^3)\big)^2-\big(u'(1)(r-1)+\frac{ u''(1)}{2}(r-1)^2+O((r-1)^3)\big)^2\Big)+O((r-1)^6)\\
&=\frac{1}{2}(u')^2(1)((v')^2(1)-(u')^2(1))(r-1)^4+\frac{\pi r^2}{2} u(1)u''(1)\big((v')^2(1)-(u')^2(1)\big)(r-1)^5\\
&\ \ +\frac{1}{2}(u')^2(1)(v'(1)v''(1)-u'(1)u''(1))(r-1)^5+O((r-1)^6)\\
&=-\frac{1}{2}(u'(1))^2((u')^2(1)-(v')^2(1))(r-1)^4+(u')^2(1)((u')^2(1)-(v')^2(1))(r-1)^5+\\&\ \ +O((r-1)^6),
\end{split}\end{equation*}
where in the last step, we have used the fact that $u''(1)+u'(1)=0$.

For $I$, we have
\begin{equation*}\begin{split}
I&=-\frac{1}{2\pi r^2}(u')^2(1)((u')^2(1)-(v')^2(1))(r-1)^4 \pi (1-r^2)+O((r-1)^6)\\
&=(u')^2(1)((u')^2(1)-(v')^2(1))(r-1)^5+O((r-1)^6).
\end{split}\end{equation*}
Combining \eqref{ad21}, \eqref{r3} and the estimates of $I$ and $II$, we derive that
$$\big(\frac{1}{2}+\frac{84}{120}\big)(u')^2(1)\big((u')^2(1)-(v')^2(1)\big)(r-1)^5=2(u')^2(1)\big((u')^2(1)-(v')^2(1)\big)(r-1)^5,$$
which is a contradiction. Hence $(c)$ is impossible. This accomplishes the proof of Theorem \ref{thm1}.
\end{proof}
\section{The proof of Theorem \ref{adthm2}}
In this section, we will prove the uniqueness of positive solution for elliptic equation in the ball of hyperbolic space and give quantization result for positive critical point of Trudinger-Moser functional in the ball of hyperbolic space.
Assume that $u_\lambda$ is a positive solution of equation \eqref{adeq3}, by using the moving plane method on hyperbolic spaces we can show that $u_{\lambda}$ must be radially symmetric  about the origin and decreasing. (see e.g., \cite{LLW}). By the conformal invariance between $(B, dV_\mathbb{H})$ and $(B, dx)$, one can rewrite the elliptic equation in the ball of hyperbolic space into the elliptic equation in ball $B_{\tilde{R}}$ of Euclidean space $\mathbb{R}^2$ with radius $\tilde{R}$ equal to $\frac{e^{R}-1}{e^{R}+1}$:

\begin{equation}\label{elliptic}
\begin{cases}
-\Delta u_{\lambda}=\lambda u_\lambda e^{u_\lambda^2}(\frac{2}{1-|x|^2})^2,\quad\quad & x\in B_{\tilde{R}}, \\
u_\lambda>0,\quad\quad & x\in B_{\tilde{R}},\\
0<\lambda<\lambda_1(B_\mathbb{H}(0,R)),\\
u_\lambda=0,\quad\quad &x\in \partial B_{\tilde{R}}.
\end{cases}
\end{equation}
 Hence $u_\lambda$ satisfies the following ODE equation
\begin{equation}\label{ode1-}
		\begin{cases}
			-(ru_\lambda')'=r\lambda u_\lambda e^{u_\lambda^2}(\frac{2}{1-r^2})^2, \quad\quad & r\in(0,\tilde{R}), \\
			u_\lambda>0,\quad\quad & r\in(0,\tilde{R}),\\
			u_\lambda'(0)=0, u_\lambda(\tilde{R})=0.
		\end{cases}
	\end{equation}
Slightly changing the proof of Theorem \ref{thm1}, one can obtain that $u_\lambda$ is unique. Hence $u_\lambda$ is also a least-energy solution of elliptic equation \eqref{elliptic}. Now we are in position to give quantization result for least energy solution of elliptic equation \eqref{elliptic}.
\medskip

We recall that $u_\lambda$ is a least energy solution of elliptic equation \eqref{elliptic} if its functional energy $I_{\lambda}(u)$ defined by $$I_{\lambda}(u)=\frac{1}{2}\int_{B_\mathbb{H}(0,R)}|\nabla_{\mathbb{H}}u|^2dV_{\mathbb{H}}-
\frac{\lambda}{2}\int_{B_\mathbb{H}(0,R)}(e^{u^2}-1)dV_{\mathbb{H}}$$ is equal to $m_\lambda:=\min\{I_{\lambda}(u):\ I_\lambda'(u)=0\}$, which is equivalent to say $I_{\lambda}(u_\lambda)=m_\lambda$. Now, we claim that $\lim\limits_{\lambda\rightarrow 0}m_{\lambda}$ has the positive lower bound. We argue this by contradiction. Suppose not, then $\lim\limits_{\lambda\rightarrow 0}m_{\lambda}=0$, which together with $I_{\lambda}'(u_\lambda)u_{\lambda}=0$ implies that $\lim\limits_{\lambda\rightarrow 0}\int_{B_\mathbb{H}(0,R)}|\nabla_{\mathbb{H}}u_{\lambda}|^2dV_{\mathbb{H}}=0$. Then it follows that $e^{u_\lambda^2}$ is bounded in $L^q(B_\mathbb{H}(0,R))$ for any $q>1$. Let $v_{\lambda}=\frac{u_{\lambda}}{\|\nabla_{\mathbb{H}}u_{\lambda}\|_2}$, then $v_\lambda$ is bounded in $W^{1,2}(B_\mathbb{H}(0,R))$. Hence, there exists $v_0\in W^{1,2}(B_\mathbb{H}(0,R))$ such that $v_\lambda$ strongly converges to $v_0$ in $L^{q}(B_\mathbb{H}(0,R))$ for any $q>1$. Noticing that $I_{\lambda}'(u_\lambda)u_{\lambda}=0$, one can write
$1=\lambda \int_{B_\mathbb{H}(0,R)}v_\lambda^2 e^{u_\lambda^2}dV_{\mathbb{H}}$.
This together with the boundedness of $v_\lambda^2$ and $e^{u_\lambda^2}$ in $L^q(B_\mathbb{H}(0,R))$ yields that
$$\lim\limits_{\lambda\rightarrow 0}\lambda \int_{B_\mathbb{H}(0,R)}v_\lambda^2 e^{u_\lambda^2}dV_{\mathbb{H}}=0,$$ which is a contradiction.
This proves that there exists $c_0>0$ such that $\lim\limits_{\lambda\rightarrow 0}\int_{B_\mathbb{H}(0,R)}|\nabla_{\mathbb{H}}u_\lambda|^2dV_{\mathbb{H}}\geq c_0$. Through $I_{\lambda}'(u_\lambda)u_\lambda=0$, we have that
\begin{equation}\label{addd3}\int_{B_\mathbb{H}(0,R)}\left\vert
\nabla_\mathbb{H} u_\lambda\right\vert ^{2}dV_{\mathbb{H}}=\lambda\int_{B_\mathbb{H}(0,R)}u_\lambda^2\exp(u_\lambda^2)dV_{\mathbb{H}},\end{equation}
which implies that \begin{equation}\label{infi}\lim\limits_{\lambda\rightarrow0}\int_{B_\mathbb{H}(0,R)}u_\lambda^2\exp(u_\lambda^2)dV_{\mathbb{H}}=+\infty.\end{equation} This deduces $c_\lambda:=\lim\limits_{\lambda\rightarrow 0} u_\lambda(0)=+\infty$, that is to say that $u_\lambda$ blows up at the origin. Now, we will prove that \begin{equation}\label{addd}\lim\limits_{\lambda\rightarrow 0}\int_{B_\mathbb{H}(0,R)}|\nabla_{\mathbb{H}}u_\lambda|^2dV_{\mathbb{H}}=4\pi.\end{equation}
Since $u_\lambda$ is the least energy critical point of functional $I_{\lambda}(u)$, one can deduce that \begin{equation}\label{rang}0<I_{\lambda}(u_\lambda)<2\pi\end{equation} by using the Trudinger-Moser inequality on compact manifold (see the Appendix). This together with $I'(u_\lambda)u_\lambda=0$ yields $u_\lambda$ is bounded in $W^{1,2}_{0}(B_\mathbb{H}(0,R))$ if $\lambda \rightarrow 0$. Then there exists some $u_0$ such that $u_\lambda\rightharpoonup u_0$ weakly in $W^{1,2}_0(B_\mathbb{H}(0,R))$ and $u_\lambda\rightarrow u_0$ strongly in $L^{q}(B_\mathbb{H}(0,R))$ for any $q>1$. Now we claim that \begin{equation}\label{lim1}\lim\limits_{\lambda\rightarrow 0}\int_{B_\mathbb{H}(0,R)}|\nabla_{\mathbb{H}}u_\lambda|^2dV_{\mathbb{H}}=\lim\limits_{\lambda\rightarrow 0}2I_\lambda(u_\lambda)\leq 4\pi.\end{equation} This is mainly because

\begin{equation}\label{4pai}\lim\limits_{\lambda\rightarrow 0}\lambda\int_{B_\mathbb{H}(0,R)}e^{u_{\lambda}^2}dV_{\mathbb{H}}=0.
\end{equation}
Indeed, we can write \begin{equation}\begin{split}
\lambda\int_{B_\mathbb{H}(0,R)}e^{u_{\lambda}^2}dV_{\mathbb{H}}&=\lambda \int_{\{|u_\lambda|>M\}}e^{u_\lambda^2}dV_{\mathbb{H}}+\lambda\int_{\{|u_\lambda|\leq M\}}e^{u_\lambda^2}dV_{\mathbb{H}}\\
&=I_1+I_2.\\
\end{split}\end{equation}
For $I_1$, through $\lambda\int_{B_\mathbb{H}(0,R)}u_{\lambda}^2e^{u_\lambda^2}dV_{\mathbb{H}}\lesssim 1$, we obtain
\begin{equation}\begin{split}
I_1\leq \frac{\lambda}{M^2}\int_{B_\mathbb{H}(0,R)}u_{\lambda}^2e^{u_\lambda^2}dV_{\mathbb{H}}\lesssim\frac{1}{M^2},
\end{split}\end{equation}
which implies $\lim\limits_{M\rightarrow \infty}\lim\limits_{\lambda\rightarrow +\infty}I_1=0$.
For $I_2$, Obviously $I_2\leq \lambda e^{M^2}\text{Vol}_{\mathbb{H}}(B_\mathbb{H}(0,R))$, hence $\lim\limits_{M\rightarrow \infty}\lim\limits_{k\rightarrow +\infty}I_2=0$. Combining the estimates of $I_1$ and $I_2$, we conclude that $$\lim\limits_{k\rightarrow+\infty}\lambda\int_{B_\mathbb{H}(0,R)}e^{u_\lambda^2}dV_{\mathbb{H}}=0,$$
hence the claim is proved.

 Now, we are in position to prove that $\lim\limits_{\lambda\rightarrow 0}\int_{B_\mathbb{H}(0,R)}|\nabla_{\mathbb{H}}u_\lambda|^2dV_{\mathbb{H}}=4\pi$. We argue this by contradiction. Suppose that $\lim\limits_{\lambda\rightarrow 0}\int_{B_\mathbb{H}(0,R)}|\nabla_{\mathbb{H}}u_\lambda|^2dV_{\mathbb{H}}<4\pi$, it follows from  Trudinger-Moser inequalities (see e.g., \cite{LT}, \cite{MS}) that $\int_{B_\mathbb{H}(0,R)}u_\lambda^2e^{u_\lambda^2}dV_{\mathbb{H}}$ is bounded, which is a contradiction with  \eqref{infi}. Hence \eqref{4pai} holds.

  Next, we claim $$|\nabla_{\mathbb{H}}u_\lambda|^2dV_{\mathbb{H}}\rightharpoonup 4\pi \delta_{0}.$$ We argue this by contradiction. Suppose not, then there exists some $\delta>0$ such that $$\int_{B_\mathbb{H}(0,\delta)}|\nabla_{\mathbb{H}}u_\lambda|^2dV_{\mathbb{H}}<4\pi.$$ Hence, through the Trudinger-Moser inequality again, we derive that
$ u_\lambda e^{u_\lambda^2} \in L^{p}(B_\mathbb{H}(0,\delta))$ for some $p>1$. Then it follows from the standard elliptic estimates that $$\sup_{\lambda}\|u_{\lambda}\|_{C^1(B_\mathbb{H}(0,\delta))}<\infty,$$ which is a contradiction with $\lim\limits_{\lambda\rightarrow +\infty}u_{\lambda}(0)=+\infty$.

\emph{The proof of (ii) of Theorem \ref{adthm2}:}
when $\lambda\rightarrow\lambda_0 \in (0, \lambda_1(B_\mathbb{H}(0,R))$, one can similarly obtain that there exist $0<\rho_1<\rho_2<2\pi$ such that   $$\rho_1\leq \lim\limits_{\lambda\rightarrow\lambda_0}I_\lambda(u_\lambda)\leq \rho_2<2\pi$$  as $\lambda\rightarrow\lambda_0$.
Gathering this and $I'(u_\lambda)u_\lambda=0$, we deduce that $u_\lambda$ is bounded in $W^{1,2}_{0}(B_\mathbb{H}(0,R)))$. Then there exists some $u_0$ such that $u_\lambda\rightharpoonup u_0$ weakly in $W^{1,2}_0(B_\mathbb{H}(0,R))$. Next, we will show that $u_0$ satisfies equation
\begin{equation}\label{limit-equation}
\begin{cases}
-\Delta_{\mathbb{H}} u_0 =\lambda_0 u_0e^{u_0^2}, \quad\quad & x\in B_\mathbb{H}(0,R), \\
u_0>0, \quad\quad & x\in  B_\mathbb{H}(0,R), \\
u_0=0, \quad\quad &x\in \partial B_\mathbb{H}(0,R).
\end{cases}
\end{equation}
For this purpose, according to the definition of weak solution, we only need to prove that $$\lim\limits_{\lambda\rightarrow \lambda_0}\lambda \int_{B_\mathbb{H}(0,R)}u_\lambda e^{u_\lambda^2}dV_{\mathbb{H}}=\lambda_0\int_{B_\mathbb{H}(0,R)}u_0e^{u_0^2}dV_{\mathbb{H}}.$$
Indeed, we can write \begin{equation}\begin{split}
\lambda\int_{B_\mathbb{H}(0,R)}u_\lambda e^{u_\lambda^2}dV_{\mathbb{H}}
&=\lambda \int_{\{|u_\lambda|>M\}}u_\lambda e^{u_\lambda^2}dV_{\mathbb{H}}+\lambda \int_{\{|u_\lambda|\leq M\}}u_\lambda e^{u_\lambda^2}dV_{\mathbb{H}}\\
&=I_1+I_2.\\
\end{split}\end{equation}
For $I_1$, through $\lambda\int_{B_\mathbb{H}(0,R)}u_{\lambda}^2e^{u_\lambda^2}dV_{\mathbb{H}}\lesssim 1$, we obtain
\begin{equation}\begin{split}
I_1\leq \frac{\lambda }{M}\int_{\mathbb{H}(0,R)}u_{\lambda}^2e^{u_\lambda^2}dV_{\mathbb{H}}\lesssim\frac{1}{M^2},
\end{split}\end{equation}
which implies $\lim\limits_{M\rightarrow \infty}\lim\limits_{\lambda\rightarrow \lambda_0}I_1=0$.
For $I_2$, through Lebesgue dominated convergence theorem, we can derive that $\lim\limits_{M\rightarrow \infty}\lim\limits_{\lambda\rightarrow \lambda_0}I_2=\lambda_0\int_{\mathbb{H}(0,R)}u_0e^{u_0^2}dV_{\mathbb{H}}$, which together with the estimate of $I_1$ gives
\begin{equation}\label{convergence}
\lim\limits_{\lambda\rightarrow \lambda_0}\lambda\int_{B_\mathbb{H}(0,R)}u_\lambda e^{u_\lambda^2}dV_{\mathbb{H}}=\lambda_0\int_{B_\mathbb{H}(0,R)}u_0\exp(u_0^2)dV_{\mathbb{H}}.
\end{equation}
Similarly, we can also prove that \begin{equation}\label{adqe}\lim\limits_{\lambda\rightarrow \lambda_0}\lambda \int_{B_\mathbb{H}(0,R)}e^{u_\lambda^2}dV_{\mathbb{H}}=\lambda_0\int_{B_\mathbb{H}(0,R)}e^{u_0^2}dV_{\mathbb{H}}.\end{equation}
Now, we claim that $u_0\neq 0$. We argue this by contradiction. If $u_0=0$, then from equality \eqref{convergence}, we see that
$\lim\limits_{\lambda\rightarrow\lambda_0}\int_{B_\mathbb{H}(0,R)}|\nabla_{\mathbb{H}}u_\lambda|^2dV_{\mathbb{H}}=2\lim\limits_{\lambda\rightarrow\lambda_0}I_\lambda(u_\lambda)\leq 2\rho_1<4\pi$. Using the Trudinger-Moser inequality and Vitali convergence theorem, we derive that
$$\lim\limits_{\lambda\rightarrow\lambda_0}\lambda\int_{B_\mathbb{H}(0,R)}u_\lambda^2e^{u_{\lambda}^2}dV_{\mathbb{H}}=\lambda_0\int_{B_\mathbb{H}(0,R)}u_0^2e^{u_0^2}dV_{\mathbb{H}}=0,$$ which implies that
$\lim\limits_{\lambda\rightarrow\lambda_0}I_\lambda(u_\lambda)=0$. This arrives a contradiction with $$0<\rho_1\leq \lim\limits_{\lambda\rightarrow\lambda_0}I_\lambda(u_\lambda).$$
So $u_0\neq 0$.

Next, we will prove that $u_\lambda\rightarrow u_0$ in $W^{1,2}_0(B_\mathbb{H}(0,R))$.
According to the convexity of norm in $W^{1,2}_0(B_\mathbb{H}(0,R))$, we only need to prove that $$\lim\limits_{\lambda\rightarrow \lambda_0}\int_{B_\mathbb{H}(0,R)}|\nabla_{\mathbb{H}}u_\lambda|^2dV_{\mathbb{H}}>\int_{B_\mathbb{H}(0,R)}|\nabla_{\mathbb{H}}u_0|^2dV_{\mathbb{H}}$$ is impossible. We argue this by contradiction. Set
$$ v_\lambda:=\frac{u_\lambda}{\lim\limits_{\lambda\rightarrow\lambda_0}\|\nabla_{\mathbb{H}} u_\lambda\|_{L^2(B_\mathbb{H}(0,R))}^2}\ \mbox{and}\ v_0:=\frac{u_0}{\lim\limits_{\lambda\rightarrow\lambda_0}\|\nabla_{\mathbb{H}} u_\lambda\|_{L^2(B_\mathbb{H}(0,R))}^2}.$$
 We claim that there exists $q_0>1$ sufficiently close to $1$ such that
 \begin{equation}\label{d.7}
  q_0\lim\limits_{\lambda\rightarrow\lambda_0}\|\nabla_{\mathbb{H}} u_\lambda\|_{L^2(B_\mathbb{H}(0,R))}^2<\frac{4\pi}{1-\|\nabla_{\mathbb{H}} v_0\|_{L^2(B_\mathbb{H}(0,R))}^2}.
  \end{equation}
 Indeed, by \eqref{adqe} and \eqref{rang}, we have
  \begin{equation}\begin{split}\label{d.8}
  &\lim\limits_{\lambda\rightarrow\lambda_0}\|\nabla_{\mathbb{H}} u_\lambda\|_{L^2(B_\mathbb{H}(0,R))}^2\big(1-\|\nabla_{\mathbb{H}} v_0\|_{L^2}^2\big)\\
  &\ \ =\lim\limits_{\lambda\rightarrow\lambda_0}\|\nabla_{\mathbb{H}} u_\lambda\|_{L^2(B_\mathbb{H}(0,R))}^2\Big(1-\frac{\|\nabla_{\mathbb{H}} u_{0}\|_{L^2(B_\mathbb{H}(0,R))}^2}{\|\nabla_{\mathbb{H}} u_\lambda\|_{L^2(B_\mathbb{H}(0,R))}^2}\Big)\\
 &\ \ =2\lim\limits_{\lambda\rightarrow \lambda_0}I_{\lambda}(u_\lambda)+\lambda\int_{B_\mathbb{H}(0,R))}(e^{u_\lambda^2}-1)dV_{\mathbb{H}}-2I_{\lambda_0}(u_{0})-\lambda_0\int_{B_\mathbb{H}(0,R)}(e^{u_0^2}-1)dV_{\mathbb{H}}\\
&\ \ <4\pi,
\end{split}\end{equation}
and then the claim is proved.  Through the concentration-compactness principle for the Trudinger-Moser inequality \cite{LLZ}, one can derive that there exists $p_0>1$ such that
\begin{eqnarray}\label{d.9}
\sup_{\lambda}\int_{B_\mathbb{H}(0,R)}\big(u_\lambda^2e^{u_\lambda^2})^{p_0}dV_{\mathbb{H}}<\infty.
\end{eqnarray}
Then it follows from  the Vitali convergence theorem that $$\lim\limits_{\lambda\rightarrow\lambda_0}\lambda\int_{B_\mathbb{H}(0,R)}u_\lambda^2\exp(u_\lambda^2)dV_{\mathbb{H}}=\lambda_0\int_{B_\mathbb{H}(0,R)}u_0^2\exp(u_0^2)dV_{\mathbb{H}}.$$
Hence, we conclude that $u_\lambda\rightarrow u_0$ in $W^{1,2}_0(B_\mathbb{H}(0,R))$ from \eqref{addd3}.
Using the Trudinger-Moser inequality in $W^{1,2}_0(B_\mathbb{H}(0,R))$, we derive that
for any $p>1$, there holds $$\int_{B_\mathbb{H}(0,R)}\big(u_\lambda^2\exp(u_\lambda^2)\big)^pdV_{\mathbb{H}}\lesssim 1.$$ Since $u_\lambda$ satisfies elliptic equation \eqref{adeq3}, standard elliptic estimate gives $u_\lambda\rightarrow u_0$ in $C^2(B_\mathbb{H}(0,R))$. Then the proof of (ii) of Theorem \ref{adthm2} is accomplished.
\medskip

\emph{The proof of (iii) of Theorem \ref{adthm2}:} We will prove that if $\lambda_0=\lambda_1(B_\mathbb{H}(0,R))$, then $u_\lambda\rightarrow 0$ in $C^{2}(B_\mathbb{H}(0,R))$. We first show that $\int_{B_\mathbb{H}(0,R)}|\nabla_{\mathbb{H}} u_\lambda|^2dV_{\mathbb{H}}$ is bounded. We argue this by contradiction. Assume that $\lim\limits_{\lambda\rightarrow \lambda_0}|\nabla u_\lambda|^2dV_{\mathbb{H}}=+\infty$, then it follows from $I_{\lambda}'(u_\lambda)u_\lambda=0$ and $I_\lambda(u_\lambda)<2\pi$ that
$$\lim\limits_{\lambda\rightarrow \lambda_0} \lambda \int_{B_\mathbb{H}(0,R)}u_\lambda^2e^{u_\lambda^2}dV_{\mathbb{H}}=+\infty,\ \ \lim\limits_{\lambda\rightarrow \lambda_0} \lambda \int_{B_\mathbb{H}(0,R)}e^{u_\lambda^2}dV_{\mathbb{H}}=+\infty.$$
On the other hand, we can also derive that
$$\lim\limits_{\lambda\rightarrow\lambda_0}\lambda \int_{B_\mathbb{H}(0,R)}\big(u_\lambda^2e^{u_\lambda^2}-(e^{u_\lambda^2}-1)\big)dV_{\mathbb{H}}=\lim\limits_{\lambda\rightarrow\lambda_0}2I_\lambda(u_\lambda)<4\pi,$$
which implies that $\int_{B_\mathbb{H}(0,R)}\big(e^{u_\lambda^2}-1-u_\lambda^2\big)dV_{\mathbb{H}}$ is bounded, hence $\int_{B_\mathbb{H}(0,R)}\big(e^{u_\lambda^2}-1\big)dV_{\mathbb{H}}$ is bounded since $\|u_\lambda\|_{L^2(B_\mathbb{H}(0,R))}\lesssim \|u_\lambda\|_{L^4(B_\mathbb{H}(0,R))}$. This arrives at a contradiction with the fact $$\lim\limits_{\lambda\rightarrow\lambda_0} \lambda \int_{B_\mathbb{H}(0,R)}\big(e^{u_\lambda^2}-1\big)dV_{\mathbb{H}}=+\infty.$$  This proves that $\int_{B_\mathbb{H}(0,R)}|\nabla_{\mathbb{H}} u_\lambda|^2dV_{\mathbb{H}}$ is bounded. Then there exists some non-negative function $u_0\in W^{1,2}_0(B_\mathbb{H}(0,R))$ such that $u_\lambda\rightharpoonup u_0$ weakly in $W^{1,2}_0(B_\mathbb{H}(0,R))$. As what we did in the previous proof of (ii), we can similarly derive that $$\lim\limits_{\lambda\rightarrow \lambda_0}\lambda\int_{B_\mathbb{H}(0,R)}u_\lambda\exp(u_\lambda^2)dV_{\mathbb{H}}=\lambda_0\int_{B_\mathbb{H}(0,R)}u_0\exp(u_0^2)dV_{\mathbb{H}}$$
and $$\lim\limits_{\lambda\rightarrow\lambda_0}\lambda\int_{B_\mathbb{H}(0,R)}e^{u_\lambda^2}dV_{\mathbb{H}}=\lambda_0\int_{B_\mathbb{H}(0,R)}e^{u_0^2}dV_{\mathbb{H}}.$$
Through equation \eqref{adeq3} and the definition of weak solution, we see that $u_0$ satisfies equation
\begin{equation}\label{limit-equation}
\begin{cases}
-\Delta_{\mathbb{H}} u=\lambda_0 ue^{u^2},\quad\quad & x\in B_1, \\
u> 0,\ & x\in B_1,\\
u=0,\quad\quad &x\in \partial B_1.
\end{cases}
\end{equation}
Noticing $\lambda_0=\lambda_1(B_\mathbb{H}(0,R))$ is the first eigenvalue of $-\Delta_{\mathbb{H}}$ in $B_\mathbb{H}(0,R)$ with the Dirichlet boundary, hence one can easily obtain $u_0=0$ through Pohozaev identity. This deduces that $$\lim\limits_{\lambda\rightarrow\lambda_0}\lambda\int_{B_\mathbb{H}(0,R)}\big(\exp(u_\lambda^2)-1)dV_{\mathbb{H}}=\lambda_0\int_{B_\mathbb{H}(0,R)}\big(\exp(u_\lambda^2)-1)dV_{\mathbb{H}}=0.$$
Hence it follows that $\lim\limits_{\lambda\rightarrow\lambda_0}\int_{B_\mathbb{H}(0,R)}|\nabla_{\mathbb{H}} u_\lambda|^2dV_{\mathbb{H}}=2\lim\limits_{\lambda\rightarrow\lambda_0}I_\lambda(u_\lambda)<4\pi$. Combining this and the Trudinger-Moser inequality, we find that there exists some $p_0>1$ such that $\int_{B_\mathbb{H}(0,R)}\big(u_\lambda^2e^{u_\lambda^2}\big)^{p_0}dV_{\mathbb{H}}$ is bounded. Using the Vitali convergence theorem, we derive that $$\lim\limits_{\lambda\rightarrow \lambda_0}\lambda\int_{B_\mathbb{H}(0,R)}u_\lambda^2\exp(u_\lambda^2)dV_{\mathbb{H}}=\lambda_0\int_{B_\mathbb{H}(0,R)}u_0^2\exp(u_0^2)dV_{\mathbb{H}}=0,$$ which implies that
\begin{equation}\label{lim2}\lim\limits_{\lambda\rightarrow\lambda_1}\int_{B_\mathbb{H}(0,R)}|\nabla_{\mathbb{H}} u_\lambda|^2dV_{\mathbb{H}}=0 \text{ and}\lim\limits_{\lambda\rightarrow\lambda_1} I_{\lambda}(u_\lambda)=0.\end{equation} That is to say $u_\lambda\rightarrow 0$ in $W^{1,2}_0(B_\mathbb{H}(0,R))$. Using the Trudinger-Moser inequality again, we derive that
for any $p>1$, there holds $\int_{B_\mathbb{H}(0,R)}\big(u_\lambda^2\exp(u_\lambda^2)\big)^pdV_{\mathbb{H}}\lesssim 1$. Since $u_\lambda$ satisfies elliptic equation \eqref{elliptic}, standard elliptic estimate gives that $u_\lambda\rightarrow u_0=0$ in $C^2(B_\mathbb{H}(0,R))$. Then the proof of (iii) of Theorem \ref{adthm2} is accomplished.
\section{The proof of Theorem \ref{adthm3}}
In this section, we will prove the multiplicity and non-existence result for the Trudinger-Moser functional $$F(u)=\int_{B_\mathbb{H}(0,R)}(e^{u^2}-1)dV_{\mathbb{H}}$$ under the constraint $\int_{B_\mathbb{H}(0,R)}|\nabla_{\mathbb{H}}u_\lambda|^2dV_{\mathbb{H}}=\gamma$ for $\gamma>4\pi$, namely we shall give the proof of Theorem \ref{adthm3}.
\medskip

Obviously, the positive critical points $u_0$ of the Trudinger-Moser functional $F(u)$ under the constraint $\int_{B_\mathbb{H}(0,R)}|\nabla_{\mathbb{H}}u_\lambda|^2dV_{\mathbb{H}}=\gamma$ must satisfy
\begin{equation}\label{eq1}
\begin{cases}
-\Delta_{\mathbb{H}} u =\lambda_0 ue^{u^2}\quad\quad & x\in B_\mathbb{H}(0,R), \\
u\geq 0,\ & x\in B_\mathbb{H}(0,R),\\
u=0, \quad\quad &x\in \partial B_\mathbb{H}(0,R),\\
\int_{B_\mathbb{H}(0,R)}|\nabla_{\mathbb{H}}u|^2dV_{\mathbb{H}}=\gamma,
\end{cases}
\end{equation}
where $\lambda_0\in (0, \lambda_1(B_\mathbb{H}(0,R)))$. Set $\Lambda_{\lambda}=\int_{B_\mathbb{H}(0,R)}|\nabla_{\mathbb{H}}u_\lambda|^2dV_{\mathbb{H}}$, where $u_\lambda$ is the positive solution of equation
\begin{equation}\label{eq2}
\begin{cases}
-\Delta_{\mathbb{H}} u_\lambda=\lambda u_\lambda e^{u_\lambda^2}\quad\quad & x\in B_\mathbb{H}(0,R), \\
u_\lambda\geq 0,\ & x\in B_\mathbb{H}(0,R),\\
u_\lambda=0, \quad\quad &x\in \partial B_\mathbb{H}(0,R).\\
\end{cases}
\end{equation}
The definition of $\Lambda_{\lambda}$ is well-defined because the positive solution of \eqref{eq2} is unique. Through Theorem \ref{adthm2}, we see that $\Lambda_{\lambda}$ is continuous with the respect to the parameter $\lambda \in (0, \lambda_1(B_\mathbb{H}(0,R)))$ and
$$\lim_{\lambda\rightarrow 0}\int_{B_\mathbb{H}(0,R)}|\nabla_{\mathbb{H}}u_\lambda|^2dV_{\mathbb{H}}=4\pi,\ \ \lim\limits_{\lambda\rightarrow \lambda_1(B_\mathbb{H}(0,R))}\int_{B_\mathbb{H}(0,R)}|\nabla_{\mathbb{H}}u_\lambda|^2dV_{\mathbb{H}}=0.$$ Hence $\Lambda_{\lambda}$ is bounded in $(0,\lambda_1(B_\mathbb{H}(0,R)))$.
Define $\gamma^{*}=\sup \{\Lambda_\lambda : \lambda \in (0,\lambda_1(B_\mathbb{H}(0,R))\}$, we see that for any $\gamma>\gamma^{*}$, Trudinger-Moser functional $F(u)$ under the constraint $\int_{B_\mathbb{H}(0,R)}|\nabla_{\mathbb{H}}u_\lambda|^2dV_{\mathbb{H}}=\gamma$ does not admit any positive critical point if $\gamma>\gamma^{*}$.
\medskip

Now, in order to finish the proof of Theorem \ref{adthm3}, we only need to prove that $F(u)$ under the constraint $\int_{B_\mathbb{H}(0,R)}|\nabla_{\mathbb{H}}u_\lambda|^2dV_{\mathbb{H}}=\gamma$ has at least two positive critical points if $\gamma\in (4\pi,\gamma^{*})$. Since $\Lambda_\lambda$ is continuous with respect to the parameter $\lambda \in (0, \lambda_1(B_\mathbb{H}(0,R)))$ and $\Lambda(0)=4\pi$, $\Lambda(\lambda_1(B_\mathbb{H}(0,R)))=0$, hence it suffices to prove that there exists a $\lambda_{*}\in (0, \lambda_1(B_\mathbb{H}(0,R)))$ such that $\Lambda(\lambda_{*})>4\pi$. This will be easily verified by showing that the Trudinger-Moser functional $F(u)$ under the constraint $\int_{B_\mathbb{H}(0,R)}|\nabla_{\mathbb{H}}u_\lambda|^2dV_{\mathbb{H}}=\gamma$ for $\gamma$ sufficiently close to $4\pi$ has a local maximum point.  The argument is essentially similar to the one for the local maximum point of the super-critical Trudinger-Moser functional on bounded domain of $\mathbb{R}^2$, which was proved by  Struwe in \cite{ST}. For simplicity, we only give the outline of the proof.
\vskip0.1cm

Step 1: Set $$\beta_{4\pi}^{*}=\sup_{\int_{B_\mathbb{H}(0,R)}|\nabla_{\mathbb{H}} u_\lambda|^2dV_{\mathbb{H}}=1}\int_{B_\mathbb{H}(0,R)}(e^{4\pi u^2}-1)dV_{\mathbb{H}},$$
then the set $$K_{4\pi}=\{u\in W_0^{1,2}(B_\mathbb{H}(0,R)):\,  \int_{B_\mathbb{H}(0,R)}|\nabla_{\mathbb{H}} u|^2dV_{\mathbb{H}}=1, \, \int_{B_\mathbb{H}(0,R)}e^{4\pi u^2}dV_{\mathbb{H}}=\beta_{4\pi}^{*} \}$$ is compact. The proof of compactness is essential to the proof of existence of extremal functional for critical Trudinger-Moser functional on two dimensional compact manifold which  is established by Y. X. Li \cite{Li, Li1}.
\medskip

Step 2: Let $\Sigma$ be the set consisting of all functions $u\in W^{1,2}_0(B_\mathbb{H}(0,R))$ satisfying $\int_{B_\mathbb{H}(0,R)}|\nabla_{\mathbb{H}} u_\lambda|^2dV_{\mathbb{H}}=1$ and
define the Dirichlet  norm neighborhoods of $K_{4\pi}$ in $\Sigma$ by
$$N_{\epsilon}=\{u\in \Sigma|\  \exists v\in K_{4\pi}\ \  s.t. \ \int_{B_\mathbb{H}(0,R)}|\nabla_{\mathbb{H}} (u-v)|^2dV_{\mathbb{H}}<\epsilon\}.$$
Similarly, we can show that for sufficiently small $\epsilon>0$, there holds
\begin{equation}
\sup_{u\in N_{2\epsilon}\setminus N_{\epsilon}}\int_{B_\mathbb{H}(0,R)}(e^{4\pi u^2}-1)dV_{\mathbb{H}}<\beta_{4\pi}^{*}=\sup_{u\in N_{\epsilon}}\int_{B_\mathbb{H}(0,R)}(e^{4\pi u^2}-1)dV_{\mathbb{H}}.
\end{equation}
\medskip

Step 3: Through compactness of $K_{4\pi}$ and uniformly local continuity of $F$, we can show that there exists $\alpha^{*}>4\pi$ and $\varepsilon>0$ such that for any $\alpha\in [4\pi, \alpha^{*})$, there holds
$$\sup_{u\in N_{2\epsilon}\setminus N_{\epsilon}}\int_{B_\mathbb{H}(0,R)}(e^{\alpha u^2}-1)dV_{\mathbb{H}}<\sup_{u\in N_{\epsilon}}\int_{B_\mathbb{H}(0,R)}(e^{\alpha u^2}-1)dV_{\mathbb{H}}=:\beta^{*}_{\alpha}$$

Combining Steps 1-3, one can easily obtain that Trudinger-Moser functional $F(u)=\int_{B_\mathbb{H}(0,R)}(e^{u^2}-1)dV_{\mathbb{H}}$ under the constraint $\int_{B_\mathbb{H}(0,R)}|\nabla_{\mathbb{H}}u_\lambda|^2dV_{\mathbb{H}}=\gamma$ for $\gamma$ sufficiently close to $4\pi$ has a local maximum point. This accomplishes the proof of Theorem \ref{adthm3}.

\section{The proof of Theorem \ref{thm5}}
Obviously, the positive critical points $u_0$ of the perturbed Trudinger-Moser functional $\tilde{H}(u)=\int_{B_1}(e^{u^2}-1-u^2)dx$ under the constraint $\int_{B_1}|\nabla u|^2dx=\beta$ must satisfy
\begin{equation}\label{eq1}
\begin{cases}
-\Delta u =\lambda_0 u(e^{u^2}-1),\quad\quad & x\in B_1 \\
u>0,\ & x\in B_1\\
u=0,\quad\quad &x\in \partial B_1\\
\int_{B_1}|\nabla u|^2dx=\beta.
\end{cases}
\end{equation}
Through the moving-plane method, it is easy to check that $u_0$ is radially decreasing. Then $u_0$ satisfies the following ODE equation
\begin{equation}\label{ode1}
		\begin{cases}
			-(ru')'=r\lambda_0 u (e^{u_\lambda^2}-1),\quad\quad & r\in(0,1), \\
			u>0,\quad\quad & r\in(0,1), \\
			u'(0)=0, u(1)=0.
		\end{cases}
	\end{equation}
Using the argument of Theorem \ref{thm1} again, we can deduce that $u_0$ is the least energy critical point of the functional
$$I_{\lambda_0}(u)=\frac{1}{2}\int_{B_1}|\nabla u|^2dx-\frac{\lambda_0}{2}\int_{B_1}(e^{u^2}-1-u^2)dx.$$ By using the Trudinger-Moser inequality and Nehari manifold method (see the Appendix), one can deduce that $I_{\lambda_0}(u_0)<2\pi$. Using $I_{\lambda_0}'(u_0)u_0=0$, we obtain
\begin{equation}\begin{split}
I_{\lambda_0}(u_0)&=\frac{\lambda_0}{2}\int_{B_1}u_0^2(e^{u_0^2}-1)dx-\frac{\lambda_0}{2}\int_{B_1}(e^{u_0^2}-1-u_0^2)dx\\
& \geq \frac{\lambda_0}{4}\int_{B_1}u_0^2(e^{u_0^2}-1)dx=\frac{1}{4}\int_{B_1}|\nabla u_0|^2dx.
\end{split}\end{equation}
This together with $I(u_0)<2\pi$ yields $\int_{B_1}|\nabla u_0|^2dx<8\pi$. Hence  the perturbed super-critical Trudinger-Moser functional $\tilde{H}(u)=\int_{B_1}(e^{u^2}-1-u^2)dx$ under the constraint $\int_{B_1}|\nabla u|^2dx=\beta$ does not admit any positive critical point if $\beta\geq 8\pi$.

\section{Appendix}
In this section, we will give the proof of functional energy $I_{\lambda}(u_\lambda)<2\pi$, where $u_\lambda$ is the least energy critical point of functional $$I_{\lambda}(u)=\frac{1}{2}\int_{B_\mathbb{H}(0,R)}|\nabla_{\mathbb{H}}u|^2dV_{\mathbb{H}}-
\frac{\lambda}{2}\int_{B_\mathbb{H}(0,R)}(e^{u^2}-1)dV_{\mathbb{H}}.$$
Recalling the definition of the least energy critical point of functional $I_{\lambda}(u)$, we know that $I_{\lambda}(u_\lambda)=m_\lambda:=\min\{I_{\lambda}(u):\ I_\lambda'(u)u=0\}$.
\vskip0.2cm

We first show that $m_{\lambda}>0$. We argue this by contradiction. Assume that $m_{\lambda}=0$, then exists a sequence $\{u_k\}_k\in W^{1,2}_0(B_\mathbb{H}(0,R))$ such that
\begin{equation*}\int_{B_\mathbb{H}(0,R)}|\nabla_{\mathbb{H}} u_k|^2dV_{\mathbb{H}}-\lambda\int_{B_\mathbb{H}(0,R)}u_k^2 e^{u_k^2}dV_{\mathbb{H}}=0,\ \forall k\in\mathbb{\mathbb{N}}\end{equation*}
and
\begin{equation*}\lim\limits_{k\rightarrow \infty}\left(\frac{1}{2}\int_{B_\mathbb{H}(0,R)}|\nabla_{\mathbb{H}} u_k|^2dV_{\mathbb{H}}-\frac{\lambda}{2}\int_{B_\mathbb{H}(0,R)}\big(e^{u_k^2}-1\big)dV_{\mathbb{H}}\right)=0.\end{equation*}
Direct computations show that
\begin{equation}\begin{split}
m_{\lambda}&=\lim\limits_{k\rightarrow \infty}\left(\frac{1}{2}\int_{B_\mathbb{H}(0,R)}|\nabla_{\mathbb{H}} u_k|^2dV_{\mathbb{H}}-\frac{\lambda}{2}\int_{B_\mathbb{H}(0,R)}(e^{u_k^2}-1)dV_{\mathbb{H}}\right)\\
&=\lim\limits_{k\rightarrow \infty}\left(\frac{\lambda}{2}\int_{B_\mathbb{H}(0,R)}u_k^2 e^{u_k^2}dV_{\mathbb{H}}-\frac{\lambda}{2}\int_{B_\mathbb{H}(0,R)}(e^{u_k^2}-1)dV_{\mathbb{H}}\right)\\
&\geq \frac{\lambda}{4}\lim\limits_{k\rightarrow \infty}\int_{B_\mathbb{H}(0,R)}u_k^2(e^{u_k^2}-1)dV_{\mathbb{H}}\\
&=\frac{1}{4}\lim\limits_{k\rightarrow \infty}\int_{B_\mathbb{H}(0,R)}\big(|\nabla_{\mathbb{H}} u_k|^2-\lambda |u_k|^2\big)dV_{\mathbb{H}}.
\end{split}\end{equation}
Since $\lambda<\lambda_1(B_\mathbb{H}(0,R))$, it follows from the Sobolev imbedding theorem that
$$u_k\rightarrow 0\ {\rm  in}\ W^{1,2}_0(B_\mathbb{H}(0,R))\ \ {\rm and}\ \ u_k\rightarrow 0\ \ {\rm in}\ L^{p}(B_\mathbb{H}(0,R))\ \ {\rm for\ any }\ p\geq 1.$$
Let $v_k=\frac{u_k}{\|\nabla_{\mathbb{H}} u_k\|_2}$, then $v_k\rightharpoonup v$ in $W^{1,2}_{0}(B_\mathbb{H}(0,R))$ with $\|v\|_2^2\leq \lim\limits\|v_k\|_2^2<\frac{1}{\lambda_1(B_\mathbb{H}(0,R))}$.

Since $u_k\rightarrow 0\ {\rm  in}\ W^{1,2}_0(B_\mathbb{H}(0,R))$, by Trudinger-Moser inequality in $W^{1,2}_{0}(B_\mathbb{H}(0,R))$, we have
$e^{u_k^2}\in L^{p}(B_\mathbb{H}(0,R))$ for any $p>1$. Then it follows from the Vitali convergence theorem that

\begin{equation}\begin{split}
1&=\lim\limits_{k\rightarrow +\infty}\int_{B_\mathbb{H}(0,R)}\frac{\lambda u_k^2}{\|\nabla_{\mathbb{H}} u_k\|_2^2}e^{u_k^2}dV_{\mathbb{H}}\\
&=\lim\limits_{k\rightarrow +\infty}\lambda\int_{B_\mathbb{H}(0,R)}e^{u_k^2}|v_k|^2dx\\
&=\lambda\|v\|_2^2\leq \frac{\lambda}{\lambda_1(B_\mathbb{H}(0,R))}<1
\end{split}\end{equation}
which is a contradiction.

Next, we start to prove that $m_{\lambda}<2\pi$. Let $w\in W^{1,2}_0(B_\mathbb{H}(0,R))$ such that $\| \nabla_{\mathbb{H}} w\|_2^2-\lambda\|w\|_2^2=1$.
Then there exists some $\gamma_{w}>0$ such that $$\int_{B_\mathbb{H}(0,R)}(|\nabla _{\mathbb{H}}\gamma_{w} w|^2-\lambda|\gamma_{w} w|^2)dV_{\mathbb{H}}-\lambda\int_{B_\mathbb{H}(0,R)}(\gamma_{w} w)^2\big(e^{(\gamma_{w} w)^2}-1\big)dV_{\mathbb{H}}=0,$$
which implies that
\begin{equation}\begin{split}
m_{\lambda}&\leq \frac{1}{2}\int_{B_\mathbb{H}(0,R)}\big(|\nabla_{\mathbb{H}}\gamma_{w} w|^2-\lambda|\gamma_{w} w|^2\big)dV_{\mathbb{H}}
-\frac{\lambda}{2}\int_{B_\mathbb{H}(0,R)}\big(e^{(\gamma_{w} w)^2}-1-(\gamma_{w} w)^2\big)dV_{\mathbb{H}}\\
&<\frac{\gamma_{w}^2}{2}\int_{B_\mathbb{H}(0,R)}\big(|\nabla_{\mathbb{H}} w|^2-\lambda|w|^2\big)dV_{\mathbb{H}}=\frac{\gamma_{w}^2}{2}.
\end{split}\end{equation}
Set $m_{\lambda}=\frac{\gamma_{\infty}^2}{2}$. Since $\big(e^{(\gamma w)^2}-1\big)w^2$ is monotone increasing about the variable $\gamma$, we derive that
\begin{equation}\begin{split}
\int_{B_\mathbb{H}(0,R)}\big(e^{(\gamma_{\infty} w)^2}-1\big)w^2dV_{\mathbb{H}}&\leq \int_{B_\mathbb{H}(0,R)}\big(e^{(\gamma_{w} w)^2}-1\big)w^2dV_{\mathbb{H}}\\
&=\int_{B_\mathbb{H}(0,R)}(|\nabla_{\mathbb{H}}w|^2-\lambda|w|^2)dV_{\mathbb{H}}=1,
\end{split}\end{equation}
which implies that $$\sup_{\int_{B_\mathbb{H}(0,R)}(|\nabla_{\mathbb{H}}  w|^2-\lambda|w|^2)dV_{\mathbb{H}}=1}\int_{B_\mathbb{H}(0,R)}\big(e^{(\gamma_{\infty} w)^2}-1\big)w^2dV_{\mathbb{H}}<\infty.$$
Noticing $$\int_{B_{\tilde{R}}(0)}|\nabla w|^2dx=\int_{B_\mathbb{H}(0,R)}|\nabla_{\mathbb{H}}  w|^2dV_{\mathbb{H}},$$
and $dV_{\mathbb{H}}=\big(\frac{2}{1-|x|^2}\big)^2dx$, we obtain that
$$\sup_{\int_{B_{\tilde{R}}(0)}|\nabla w|^2dx=1}\int_{B_{\tilde{R}}(0)}\big(e^{(\gamma_{\infty} w)^2}-1\big)w^2dx<\infty,$$
where $B_{\tilde{R}}(0)$ denotes the ball with radius $\tilde{R}$ equal to $\frac{e^{R}-1}{e^{R}+1}$ in $\mathbb{R}^2$.
Then one can construct well-known Moser sequence (that is a concentration sequence which blows up at some point) to deduce that $m_{\lambda}=\frac{\gamma_{\infty}^2}{2}<2\pi$.

\end{document}